\documentclass[
  sn-mathphys-num,
  colorlinks,
  linkcolor=darkmagenta,
  citecolor=rufous,
  urlcolor=zaffre,
  hypertexnames=false
]{sn-jnl}

\usepackage[hyperpageref]{backref} 

\usepackage{graphicx}%
\usepackage{multirow}%
\usepackage{amsmath,amssymb,amsfonts}%
\usepackage{amsthm}%
\usepackage[title]{appendix}%
\usepackage{xcolor}%
\usepackage{textcomp}%
\usepackage{manyfoot}%
\usepackage{xurl} 
\usepackage{booktabs}%
\usepackage{algorithm}%
\usepackage{algorithmicx}%
\usepackage{algpseudocode}%
\usepackage{listings}%
\usepackage{paralist}
\usepackage[misc]{ifsym}
\usepackage{multicol}
\usepackage{tikz}
\usepackage{makecell}
\usepackage{epsfig}
\usepackage{lmodern}

\usepackage{mathtools} \mathtoolsset{showonlyrefs}

\numberwithin{equation}{section}
\newtheorem{theorem}{Theorem}[section]

\newtheorem{lemma}[theorem]{Lemma}%
\newtheorem{remark}[theorem]{Remark}%

\newtheorem{definition}{Definition}[section]%
\newtheorem{corollary}[theorem]{Corollary}
\newtheorem{assumption}{A}

\begin{document}
\title[Regularity and singularity of the blow-up]{Regularity and singularity of the blow-up curve for a wave equation with a derivative nonlinearity and a scale-invariant damping}

\author[1]{\fnm{Ahmed} \sur{Bchatnia}}\email{ahmed.bchatnia@fst.utm.tn}
\equalcont{These authors contributed equally to this work.}

\author[2]{\fnm{Makram} \sur{Hamouda}}\email{mmhamouda@iau.edu.sa}
\equalcont{These authors contributed equally to this work.}

\author[1]{\fnm{Firas} \sur{Kaabi}}\email{firaskaabi17@gmail.com}
\equalcont{These authors contributed equally to this work.}

\author[3]{\fnm{Takiko} \sur{Sasaki}}\email{t-sasaki@musashino-u.ac.jp}
\equalcont{These authors contributed equally to this work.}

\author*[4]{\fnm{Hatem} \sur{Zaag}}\email{hatem.zaag@math.cnrs.fr}
\equalcont{These authors contributed equally to this work.}

\affil[1]{\orgdiv{Department of Mathematics}, \orgname{Faculty of Sciences of Tunis}, \orgaddress{\street{University of Tunis El Manar, LR Analyse Non-Lin\'eaire et G\'eom\'etrie, LR21ES08}, \city{El Manar 2}, \postcode{2092}, \state{Tunis}, \country{Tunisia}}}

\affil[2]{\orgdiv{Department of Basic Sciences}, \orgname{Deanship of Preparatory Year and Supporting Studies}, \orgaddress{\street{Imam Abdulrahman Bin Faisal University}, \postcode{P.O. Box 1982}, \state{Dammam}, \country{Saudi Arabia}}}

\affil[3]{\orgdiv{Department of Mathematical Engineering}, \orgname{Faculty of Engineering, Musashino University}, \orgaddress{\street{3--3--3 Ariake}, \postcode{Koto--ku}, \state{Tokyo 135--8181}, \country{Japan}}}

\affil*[4]{\orgdiv{Université Sorbonne Paris Nord}, \orgname{LAGA}, \orgaddress{\street{CNRS (UMR 7539)}, \postcode{F--93430}, \state{Villetaneuse}, \country{France}}}

\abstract{
In this article, we investigate the blow-up behavior of solutions to the one-dimensional damped nonlinear wave equation, namely
$$
\partial_t^2 u - \partial_x^2 u + \frac{\mu}{1 + t} \partial_t u = |\partial_t u|^p \quad (p > 1).
$$
Under the assumption of sufficiently large and smooth initial data, we establish that the blow-up curve is continuously differentiable ($\mathcal{C}^1$). A key step in our analysis involves the characterization of the blow-up profile of the solution. The proof relies on transforming the equation into a first-order system and adapting the techniques of  Sasaki in \cite{Sasaki2018,Sasaki2019} which have elegantly extended  the method of Caffarelli and Friedman \cite{Caffarelli1986} to nonlinear wave equations with time derivative nonlinearity, but without the scale-invariant term ($\mu =0$).
}




\maketitle

\section{Introduction}
We consider the nonlinear wave equation with damping:

\begin{equation}
\label{1.1}
\begin{cases}
\partial_t^2 u - \partial_x^2 u + \dfrac{\mu}{1 + t} \partial_t u = |\partial_t u|^p, & x \in \mathbb{R}, \; t > 0, \\
u(x, 0) = u_0(x), & x \in \mathbb{R}, \\
\partial_t u(x, 0) = u_1(x), & x \in\mathbb{R},
\end{cases}
\end{equation}
where $p > 1$ is a constant and the nonlinearity $s \mapsto s^p$ belongs to $\mathcal{C}^4$ for $s \geq 0$.

It is well-known \cite[Theorem 2.3]{HH2021} that the blow-up for \eqref{1.1} occurs, at least for small data, for the following values of the nonlinearity exponent $p$:
\begin{equation}\label{p-range}
    1< p < p_G(1, \mu) := 1 + \frac{2}{\mu}, \quad \mu>0,
\end{equation}
where $p_G(1, \mu)$ denotes the Glassey exponent, which is a good candidate for characterizing the threshold between blow-up and global existence of \eqref{1.1} in dimension $N=1$. Note that for higher dimensions this ``critical'' exponent is given by $\displaystyle p_G(N, \mu) := 1 + \frac{2}{N-1+\mu}$; see e.g. \cite{HH2021}. 
Throughout this article, we assume that condition \eqref{p-range} is met, in line with the objectives of this study.

Let $R^*$ and $T^*$ be arbitrary positive constants. We define the following sets:
\begin{align}
B_{R^*} &:= \{ x \in \mathbb{R} : |x| < R^* \}, \label{eq:ball} \\
K_{-}(x_0, t_0) &:= \{ (x, t) \in \mathbb{R} \times (0, \infty) : |x - x_0| < t_0 - t \}, \label{eq:cone} \\
K_{R^*, T^*} &:= \bigcup_{x \in B_{R^*}} K_{-}(x, T^*). \label{eq:union}
\end{align}
The blow-up time function is defined by
\begin{equation}
T(x) := \sup \{ t \in (0, T^*) : |\partial_t u(x, t)| < \infty \}, \quad x \in B_{R^*}. \label{eq:blowup_time}
\end{equation}
In this paper, we refer to the set 
$$
\Gamma = \{(x, T(x)) : x \in B_{R^*}\},
$$ 
as the \textit{blow-up curve}. For notational convenience, we shall identify $\Gamma$ with the function $T$ itself. \\
This work has two primary objectives: first, to establish the continuous differentiability ($\mathcal{C}^1-$regularity) of $T$ for appropriately chosen initial conditions. We now review relevant analytical results from prior studies of blow-up curves in nonlinear wave equations. Most existing research has focused on equations of the form
$$
\partial_t^2 u - \partial_x^2 u = F(u), \quad x \in \mathbb{R},\ t > 0,
$$
with the associated blow-up curve defined by
$$
\hat{T}(x) = \sup \{ t \in (0, T^*) : |u(x,t)| < \infty \}.
$$  
We emphasize that our definition of the blow-up curve \eqref{eq:blowup_time} differs from these classical formulations. The starting point of this line of research is the work of Caffarelli and Friedman \cite{Caffarelli1985,Caffarelli1986}, who studied the nonlinear wave equation $\square u = u^{p}$ and proved that, under suitable assumptions on the data, the blow-up set is a $\mathcal{C}^{1}$ space-like surface, which reduces to a $\mathcal{C}^{1}$ blow-up curve $\hat T$ in  one dimensional space. Their approach relies on the positivity of the fundamental solution for the linear wave operator and on a comparison to the associated ODE $w'' = w^{p}$. This yields  bounds on $u$ and its derivatives in terms of the distance to the blow-up boundary. A blow-up scaling combined with the maximum principle then shows that the limiting blow-up cones are flat, and this geometrical property translates into the $\mathcal{C}^{1}-$regularity of $\hat T$.

Later, Godin \cite{Godin2002} studied the one–dimensional Liouville equation $u_{tt}-u_{xx}=e^{u}$ and gave a detailed description of the $\mathcal{C}^\infty$ and analytic regularity of the blow-up curve. Building on the earlier $\mathcal{C}^1-$theory, he showed that the blow-up curve is actually $\mathcal{C}^\infty$ and that, for mixed boundary problems, its analyticity near the boundary depends, in a subtle and discrete way, on the choice of boundary conditions. In the presence of Dirichlet boundary conditions, a refined description of the blow-up curve for one-dimensional nonlinear wave equations was obtained by Ishiwata and Sasaki \cite{IshiwataSasakiDirichlet2020}. The study of blow-up curves was extended to systems of nonlinear wave equations by Uesaka \cite{Uesaka2009}.
 More recently, Sasaki revisited this program for the derivative-type nonlinearity
$\square u = |u_t|^{p}$  in a series of works  \cite{Sasaki2018,Sasaki2019}.
Building on the ideas of Caffarelli and Friedman, Sasaki introduced the characteristic–ODE method, which leads to a modified dynamical
system along the rays $x \pm t=c$ and allows the study of the blow-up geometry in this new setting. In addition, for semilinear equations of the form $\square u = F(u)$, Sasaki established convergence and further regularity properties of the blow-up curve in \cite{SasakiConvergence2021,SasakiRegularity2022}.

In contrast to these smoothness results, Merle and Zaag \cite{Merle2005,Merle2012} demonstrated that singular points can occur in blow-up curves, revealing the more complex behavior possible in these systems. We also point out recent refinements of the blow-up rate for related semilinear wave equations, both in the non-scaling invariant and in the superconformal regimes; see Hamza and Zaag \cite{HamzaZaagNonScaling2021,HamzaZaagSuperconformal2025}. In addition, a number of related blow-up and lifespan results have been obtained for damped wave-type models with derivative or mixed nonlinearities. 
For weakly coupled wave equations with scale-invariant damping and derivative-type terms, sharp blow-up and lifespan estimates were derived  in Hammouda and Hamza works; see \cite{MH2,MH4}. Their works also extend to the generalized Tricomi equation with mixed nonlinearities where 
Blow-up and lifespan estimates were established in \cite{MH55}, while \cite{HHP2} treats a damped wave equation in the Einstein--de Sitter spacetime with a derivative-type nonlinearity. 
These works further highlight how time-dependent damping and derivative nonlinearities influence the blow-up and the lifespan, and they are complementary to our present focus on the geometry and regularity of the blow-up curve.

As mentioned above, several results have been established on $\hat{T}$ when there are no nonlinear terms involving the derivative of the solution. On the other hand, to the best of our knowledge, only two results have been found concerning $\hat{T}$ with nonlinear terms involving the derivative of the solution.

Ohta and Takamura \cite{Ohta1998} investigated the nonlinear wave equation
\begin{equation}
\partial_t^2 u - \partial_x^2 u = (\partial_t u)^2 - (\partial_x u)^2, \quad x \in \mathbb{R},\ t \in \mathbb{R}.
\label{eq:Ohta}
\end{equation}
The above equation admits an exact linearization through the transformation $v(x,t) = \exp\{-u(x,t)\}$, \textit{i.e.}, $u(x,t) = -\log v(x,t)$, which reduces it to the classical wave equation $\partial_t^2 v - \partial_x^2 v = 0$. Although this linearization enables a complete analysis of the blow-up curve for \eqref{eq:Ohta}, this approach is not applicable to our main equation \eqref{1.1}.

A more general framework was developed by Hamza and Zaag \cite{HamzaZaag2013} for equations of the form,
$$
\partial_t^2 u - \partial_x^2 u = |u|^{p-1}u + F_1(u) + F_2\left(x, t, \partial_x u, \tfrac{x}{|x|} \cdot \partial_t u\right), \quad x \in \mathbb{R},\ t \in \mathbb{R},
$$
where $F_1 \in \mathcal{C}^1(\mathbb{R},\mathbb{R})$ and $F_2 \in \mathcal{C}^1(\mathbb{R}^4,\mathbb{R})$ satisfy specific growth conditions. First, $F_1$ satisfies $|F_1(u)| \leq M(1 + |u|^q)$ for some $M > 0$ and $q < p$. Second, $F_2$ satisfies $|F_2(x,t,v,z)| \leq M(1 + |v| + |z|)$ for some $M > 0$. The techniques from \cite{HamzaZaag2013} cannot be directly applied to \eqref{1.1} because our nonlinear term involving $\partial_t u$ fails to satisfy global Lipschitz conditions.

Regarding maximal existence time for \eqref{1.1}, Haraux \cite{Haraux1992} established the existence of unbounded global solutions, while Messikh \cite{Messikh2011} proved that for any $T > 0$, there exist initial data inducing blow-up at $t^* \leq T$. For related blow-up results in the one-dimensional quasilinear setting, we refer to Haruyama and Takamura \cite{HaruyamaTakamuraQuasilinear2025}.
However, the blow-up curve for \eqref{1.1} remains uncharacterized in the literature.  Now, we are considering a new program, where we will be closely following Takiko’s work, with more adaptations.

Unlike the case of the semilinear wave equation $\Box u = |u|^{p-1}u$, we have no variational structure in our PDE. For that reason, the strategy developed by Merle and Zaag in a series of papers \cite{MZajm03, MZimrn05, MZjfa07, MZcmp08} and which extensively uses energy estimates breaks down in our case. Fortunately, we could successfully implement in our case the strategy of Caffarelli and Friedman in \cite{Caffarelli1985} and \cite{Caffarelli1986} and its adaptation by Sasaki in \cite{Sasaki2018, Sasaki2019}, which does not rely on energy estimates.

The proof strategy which led to our primary results is reasoned out in the following manner. The method we utilized is inspired by the works of \cite{Caffarelli1985, Caffarelli1986, Sasaki2019}. To begin with, there is a change of variables such that the equation is expressed in a suitable way in terms of the characteristics. In the absence of damping, this amounts to treating the problem as a system of ordinary differential equations along the characteristic rays 
$x\pm t=c$, which is the starting point of Sasaki's work \cite{Sasaki2018, Sasaki2019}. By this method, we are able to derive an associated system of ODEs that captures the geometry of the blow-up curve in terms of the dynamics of an effective one-dimensional model.

At this juncture, a number of steps follow closely on Sasaki's proof-steps; the most significant one being the characteristic reformulation. The major obstacle, however, is the damping term with coefficient $\mu$, which besides complicating the situation, adds more terms. As a result, the system that we obtain is dissimilar to that in \cite{Sasaki2019}; in particular, the equations along the rays receive an extra $\mu-$dependent term that not only destroys but also complicates the comparison arguments (see e.g. system~\eqref{1.5} below).

In order to overcome this barrier, we put forward some additional structural conditions (see \eqref{condgam}, A\ref{A4}, and A\ref{A5} below) and we resort to a number of technical reductions at different stages of the analysis. The use of these tools permits us, whenever the need arises, to recreate a scenario that is still comparable to the undamped case considered by Sasaki in \cite{Sasaki2019}.

The novel contributions of this paper are twofold. First, we prove $\mathcal{C}^1$ regularity of the blow-up curve for \eqref{1.1}. Second, as a crucial intermediate result, we identify and characterize the blow-up profile of solutions. In fact, we show that the blow-up profile for $u_t$ is approximately of the form
$$
\partial_t u(x,t) \sim \left[ \frac{1}{(p-1)(T-t)} \right]^{1/(p-1)} \quad \text{(main singularity)}.
$$
modulated by a function of $x$ and $t$, but the time-dependent damping $\frac{\mu}{1+t}$ is slowly varying, so near blow-up time $T$ it is roughly constant $\approx \frac{\mu}{1+T}$, thus the blow-up rate is the same as in the constant damping case: $(T-t)^{-1/(p-1)}$.

The spatial profile is expected to be approximately self-similar of the second kind (or peaking) if $\mu$ small, but for large $\mu$ it may be flatter.

The remainder of this article is organized as follows. Section \ref{sec2} presents preliminary computations and establishes key technical lemmas. In Section \ref{sec3}, we state and prove our main result on the $\mathcal{C}^1$ regularity of the blow-up curve. Section \ref{sec4} is devoted to the existence and regularity of solutions to our transformed system, while Section \ref{sec5} analyzes the blow-up rates and establishes Lipschitz continuity of the blow-up curve. The blow-up limits and their properties are examined in Section \ref{sec6}, leading to the proof of differentiability in Section \ref{sec7}. Finally, Section~\ref{section8} establishes the $\mathcal{C}^1$-regularity of the blow-up curve.

\section{Preliminary computations}\label{sec2}
 We transform the original equation \eqref{1.1} into a first-order system that eliminates derivatives from the nonlinear terms, adapting the characteristic method of Caffarelli and Friedman \cite{Caffarelli1985,Caffarelli1986}. Introducing the Riemann invariants:

\begin{equation}
\phi = \partial_t u + \partial_x u, \quad \psi = \partial_t u - \partial_x u,
\end{equation}
we obtain the equivalent system:
\begin{equation}
\label{1.5}
\begin{cases}
D_- \phi = 2^{-p}|\phi + \psi|^p - \dfrac{\mu}{1 + t} \dfrac{\phi + \psi}{2}, & x \in \mathbb{R}, t > 0, \\
D_+ \psi = 2^{-p}|\phi + \psi|^p - \dfrac{\mu}{1 + t} \dfrac{\phi + \psi}{2}, & x \in \mathbb{R}, t > 0, \\
\phi(x, 0) = f(x), \quad \psi(x, 0) = g(x), & x \in \mathbb{R},
\end{cases}
\end{equation}
where the differential operators $D_- = \partial_t - \partial_x$ and $D_+ = \partial_t + \partial_x$ represent characteristic derivatives, and the initial data is given by $f = u_1 + \partial_x u_0$, $g = u_1 - \partial_x u_0$. This formulation moves the nonlinearity from the derivative term $\partial_t u$ in \eqref{1.1} to the undifferentiated terms $\phi$ and $\psi$, while maintaining the damping effect through the coefficient $\mu/(1+t)$.

Consider the coupled ODE system for $(\hat{\phi}, \hat{\psi})$:

\begin{equation}
\label{1.6}
\begin{cases}
\dfrac{d\hat{\phi}}{dt} = 2^{-p}|\hat{\phi} + \hat{\psi}|^p - \dfrac{\mu}{2}(\hat{\phi} + \hat{\psi}), & t > 0, \\[8pt]
\dfrac{d\hat{\psi}}{dt} = 2^{-p}|\hat{\phi} + \hat{\psi}|^p - \dfrac{\mu}{2}(\hat{\phi} + \hat{\psi}), & t > 0, \\[8pt]
\hat{\phi}(0) = \gamma_1, \quad \hat{\psi}(0) = \gamma_2.
\end{cases}
\end{equation}
The behavior of solutions to our PDE system \eqref{1.5} can be understood through analysis of the simpler ODE system \eqref{1.6}, which captures the essential nonlinear dynamics. The following lemma establishes the critical threshold for finite-time blow-up in this reduced system:
\begin{lemma}[Finite-time blow-up]\label{lem21}
For initial data satisfying $\gamma_1 + \gamma_2 > A$, where $A = 2\mu^{1/(p-1)}$, the solution of \eqref{1.6} exhibits a finite-time blow-up, that is, there exists $T_1 = T_1(\gamma_1,\gamma_2) > 0$ such that $$\lim\limits_{t\to T_1^-} (\hat{\phi}(t) + \hat{\psi}(t)) = \infty.$$
\end{lemma}
\begin{proof}
The proof proceeds through careful analysis of the summed ODEs. First, adding both equations in \eqref{1.6} yields:

$$
\frac{d}{dt}(\hat{\phi} + \hat{\psi}) = 2^{1-p}|\hat{\phi} + \hat{\psi}|^p - \mu(\hat{\phi} + \hat{\psi}).
$$

Letting $y(t) = \hat{\phi}(t) + \hat{\psi}(t)$, we obtain the Bernoulli equation:

\begin{equation}
\frac{dy}{dt} = 2^{1-p}y^p - \mu y.
\label{bern_eq}
\end{equation}
The substitution $v = y^{1-p}$ converts \eqref{bern_eq} into a linear equation. A simple differentiation gives
$$
\frac{dv}{dt} = (1-p)y^{-p}\frac{dy}{dt} = (1-p)(2^{1-p} - \mu y^{1-p}).
$$
Since $v = y^{1-p}$, this becomes:
$$
\frac{dv}{dt} + (1-p)\mu v = (1-p)2^{1-p}.
$$
The integrating factor method yields the solution:
$$
v(t) = \frac{2^{1-p}}{\mu} + Ce^{-(1-p)\mu t},
$$
where the constant $C$ is determined by initial conditions. Substituting back for $y$, we obtain
$$
y(t) = \left[ \left( (\gamma_1+\gamma_2)^{1-p} - \frac{2^{1-p}}{\mu} \right) e^{-(1-p)\mu t} + \frac{2^{1-p}}{\mu} \right]^{\frac{1}{1-p}}.
$$
The blow-up occurs when the denominator vanishes, giving the finite blow-up time:
$$
T_1 = \frac{1}{(p-1)\mu} \ln\left( \frac{2^{1-p}}{\mu\left[\frac{2^{1-p}}{\mu} - (\gamma_1+\gamma_2)^{1-p}\right]} \right).
$$

This completes the proof, establishing that $y(t) \to +\infty$ as $t \to T_1^-$ when $\gamma_1 + \gamma_2 > 2\mu^{1/(p-1)}$.
\end{proof}

Now, we analyze the characteristic structure of the damped wave equation:
\begin{equation} \label{eq:main}
    \partial_t^2 \phi - \partial_x^2 \phi + \frac{\mu}{1 + t} \partial_t \phi = 0.
\end{equation}

Introducing $u = \partial_t \phi$ and $v = \partial_x \phi$, we obtain:
\begin{align}
    \partial_x u - \partial_t v &= 0 \label{eq:ux-vt} \\
    \partial_t u - \partial_x v &= -\frac{\mu}{1 + t} u \label{eq:system}
\end{align}
The above system has the following symmetric hyperbolic form:
\begin{equation} \label{eq:matrixform}
\partial_t U + A \partial_x U + B U = 0, \quad 
U = \begin{pmatrix} u \\ v \end{pmatrix}, \quad
A = \begin{pmatrix} 0 & -1 \\ -1 & 0 \end{pmatrix}, \quad
B = \begin{pmatrix} \frac{\mu}{1 + t} & 0 \\ 0 & 0 \end{pmatrix}.
\end{equation}

A straightforward characteristic analysis, using the transformation $V = \left((1+t)^\mu u, v \right)^t$, yields the following:
\begin{equation} \label{eq:transformed}
\partial_t V + A \partial_x V = 0.
\end{equation}
Finally, we conclude that the characteristic slopes, coming from $\det(A - \alpha I) = \alpha^2 - 1 = 0$, gives two characteristic families:
\begin{itemize}
    \item Right-going characteristics: $x - t = \text{constant}$ ($\alpha=+1$).
    \item Left-going characteristics: $x + t = \text{constant}$ ($\alpha=-1$).
\end{itemize}
In fact, the damping term $\frac{\mu}{1+t}u_t$ affects only the wave amplitudes along these characteristics.

\section{Main results}\label{sec3}
For our analysis, we consider positive constants $\gamma_1, \gamma_2$ satisfying
\begin{equation}
\label{condgam}
\gamma_1 + \gamma_2 \;>\; \max\!\left(1,\; (\mu\,p\,2^{p})^{\frac{1}{p-1}}\right).
\end{equation}

Throughout this article, we shall assume the following assumptions:
\begin{assumption}\label{A1}
$T_1 < T^*$.
\end{assumption}

\begin{assumption}\label{A2} $f \geq \gamma_1$ and $g \geq \gamma_2$ in $B_{R^* + T^*}$.
\end{assumption}

\begin{assumption}\label{A3} $f, g \in \mathcal{C}^4(\overline{B_{R^* + T^*}})$.
\end{assumption}

\begin{assumption}\label{A4} There exists $\varepsilon_0 > 0$ such that
$$
2^{-p}(\gamma_1 + \gamma_2)^p - \frac{\mu}{2}(\gamma_1 + \gamma_2) \geq (2 + \varepsilon_0) \max_{x \in B_{R^*+T^*}} (|f^\prime(x)| + |g^\prime(x)|).
$$
\end{assumption}
\begin{assumption}\label{A5} There exists $\varepsilon_1 \neq -1$ verifying
\(
2^{-p}(2p-1)\!\left(1+\frac{\varepsilon_1}{2(1+\varepsilon_1)}\right)\!\left(1-\frac{1}{2p}\right)
-2^{-p+1}p-\mu >0
\)
such that
$$
2^{-p}(\gamma_1 + \gamma_2)^p - \frac{\mu}{2}(\gamma_1 + \gamma_2) 
\;\geq\; (2+\varepsilon_1)\max_{x \in B_{R^*+T^*}} \bigl(|f^\prime(x)| + |g^\prime(x)|\bigr).
$$
\end{assumption} 

Before presenting our main results, we provide some remarks on the hypotheses introduced above.
\begin{remark}
These assumptions are consistent and can be satisfied through the following procedure:
\begin{enumerate}
    \item Choose $\gamma_1, \gamma_2 > 0$ with $\gamma_1 + \gamma_2 > 2\mu^{1/(p-1)}$ (satisfying A\ref{A1}).
    \item Select $f,g \in \mathcal{C}^4(\overline{B_{R^* + T^*}})$ with $f \geq \gamma_1$, $g \geq \gamma_2$ (satisfying A\ref{A2}--A\ref{A3}).
    \item Adjust $\gamma_1, \gamma_2$ if needed to satisfy A\ref{A5} while maintaining A\ref{A1}.
    \item Modify $f,g$ to preserve A\ref{A2} without changing their derivative bounds.
\end{enumerate}
This ensures all conditions A\ref{A1}-A\ref{A5} hold simultaneously.
\end{remark}
Having established the preliminary framework, we now present our central findings. The first and most important result characterizes the blow-up curve.

\begin{theorem}[Existence and regularity of blow-up curve]
\label{thm31}
Let $R^*$ and $T^*$ be arbitrary positive numbers. Under the assumptions \ref{A1}-$(A_5)$, there exists a unique function $T \in \mathcal{C}^1(B_{R^*})$ with $0 < T(x) < T^*$ for all $x \in B_{R^*}$, and a unique solution $(\phi, \psi) \in (\mathcal{C}^{3,1}(\Omega))^2$ to \eqref{1.5} satisfying
\begin{equation}
\phi(x,t), \psi(x,t) \to \infty \quad \text{as } t \nearrow T(x),
\label{eq:blowup}
\end{equation}
for each $x \in B_{R^*}$, where $\Omega = \{(x,t) \in \mathbb{R}^2 : x \in B_{R^*}, 0 < t < T(x)\}$.
\end{theorem}

\begin{remark}
\rm The system \eqref{1.5} is equivalent to the original equation \eqref{1.1} through the transformation
\begin{equation}
    \label{equiv}
u(x,t) = u_0(x) + \frac{1}{2}\int_0^t (\phi + \psi)(x,s) ds,
\end{equation}
and the blow-up conclusion \eqref{eq:blowup} implies that $\partial_t u(x,t) \to \infty$ as $t \nearrow T(x)$. Furthermore, thanks to \eqref{equiv}, the next results, in Theorems \ref{thm3.3}, \ref{thm3.2},  \ref{th:3.6} and \ref{thm3.4} below, apply as well for $u$, the solution of \eqref{1.1}.
\end{remark}

Building on Theorem \ref{thm31}, we establish additional properties of the blow-up curve.
\begin{theorem}[Blow-up rates]
\label{thm3.3}
Let $\phi$ and $\psi$ solutions of \eqref{1.5}. Then, under the assumptions A\ref{A1}--A\ref{A4}, there exist positive constants $C_1, C_2 > 0$ (depending only on $p$ and $\varepsilon_0$) such that the following estimates hold throughout $\Omega$ (defined by \eqref{omega}) with $q = \frac{1}{p-1}$:
\begin{align}
C_1 (\phi + \psi)^p(x,t) &\leq \partial_t \phi(x,t) \leq C_2 (\phi + \psi)^p(x,t), \label{aa} \\
C_1 (T(x) - t)^{-q-1} &\leq \partial_t \phi(x,t) \leq C_2 (T(x) - t)^{-q-1}, \label{bb} \\
C_1 (\phi + \psi)^p(x,t) &\leq \partial_t \psi(x,t) \leq C_2 (\phi + \psi)^p(x,t), \label{cc} \\
C_1 (T(x) - t)^{-q-1} &\leq \partial_t \psi(x,t) \leq C_2 (T(x) - t)^{-q-1}, \label{dd} \\
C_1 (T(x) - t)^{-q} &\leq (\phi + \psi)(x,t) \leq C_2 (T(x) - t)^{-q}. \label{ee}
\end{align}
\end{theorem}

\begin{theorem}[Lipschitz continuity]
\label{thm3.2}
The blow-up curve from Theorem \ref{thm31} satisfies the Lipschitz condition
\begin{equation}\label{eq:3.15}
|T(x) - T(y)| \leq \frac{1}{1+\varepsilon_0}|x-y| \quad \forall \ x,y \in B_{R^*},
\end{equation}
where $\varepsilon_0$ is involved in $(A_4)$.
\end{theorem}
Having established the regularity properties of the blow-up curve, we now characterize its specific form under the full set of assumptions:

\begin{theorem}[Linear structure of the limiting blow-up curve]
\label{th:3.6}
Under the assumptions A\ref{A1}--A\ref{A5}, there exists $\alpha \in (-1,1)$ such that:
\begin{align}
T_0(y) &= \alpha y \quad \text{for all } y \in \mathbb{R}, \label{eq:5.4} \\
v_\phi &= V_{\phi,\alpha}, \quad v_\psi = V_{\psi,\alpha}, \label{eq:5.5}
\end{align}
where $v_\phi$ and $v_\psi$ represent the self-similar profiles associated with the solutions $\phi$ and $\psi$,  respectively.
\end{theorem}
As a consequence of Theorems~\ref{thm3.3}, \ref{thm3.2} and~\ref{th:3.6}, we obtain the $\mathcal{C}^1-$regularity of the blow-up curve in the following result.
\begin{theorem}[$\mathcal{C}^1-$regularity of the blow-up curve]
\label{thm3.4}
The blow-up curve $x \mapsto T(x)$ is of class $\mathcal{C}^1$ in $B_{R^*}$.
\end{theorem}

\section{Proof of Theorem \texorpdfstring{\ref{thm31}}{Theorem 1}: Construction via Successive Approximations}\label{sec4}
In this section, we aim to prove the existence and regularity of solutions to system \eqref{1.5}, as stated in Theorem \ref{thm31}. Our approach is based on an iterative scheme. We start with the constant initial functions $\phi_0(x,t) \equiv \gamma_1$ and $\psi_0(x,t) \equiv \gamma_2$, as given in \eqref{1.6}, and then construct recursively the sequences $\{\phi_n\}_{n=0}^\infty$ and $\{\psi_n\}_{n=0}^\infty$ using the characteristic system:
\begin{equation}
\label{2.1}
\begin{cases}
D_- \phi_{n+1} = 2^{-p}|\phi_n + \psi_n|^p - \dfrac{\mu}{1 + t} \dfrac{\phi_n + \psi_n}{2}, & (x,t) \in \mathbb{R} \times \mathbb{R}^+, \\
D_+ \psi_{n+1} = 2^{-p}|\phi_n + \psi_n|^p - \dfrac{\mu}{1 + t} \dfrac{\phi_n + \psi_n}{2}, & (x,t) \in \mathbb{R} \times \mathbb{R}^+, \\
\phi_{n+1}(x, 0) = f(x), \quad \psi_{n+1}(x, 0) = g(x), & x \in \mathbb{R},
\end{cases}
\end{equation}
for all $n \geq 0$. This system admits an equivalent integral formulation, which we will use to carry out our analysis, as follows:
\begin{equation}
\label{2.2}
\begin{aligned}
\phi_{n+1}(x,t) &= f(x+t) + \int_0^t \mathcal{N}_n(x+t-s,s)\,ds, \\
\psi_{n+1}(x,t) &= g(x-t) + \int_0^t \mathcal{N}_n(x-t+s,s)\,ds,
\end{aligned}
\end{equation}
where $\mathcal{N}_n(x,s) := 2^{-p}|\phi_n(x,s) + \psi_n(x,s)|^p - \dfrac{\mu}{1+s}\dfrac{\phi_n(x,s) + \psi_n(x,s)}{2}$ represents the nonlinear coupling term.\\
The integral formulation reveals the characteristic structure of the problem; $\phi_n$ propagates right along the lines $x+t = \text{constant}$ while $\psi_n$ propagates left along the lines $x-t = \text{constant}$, with nonlinear coupling through the interaction terms. Taking advantage of the damping effects from the $\mu/(1+t)$ terms and the controlled growth of the nonlinearities under our assumptions, the convergence of these iterative sequences and the regularity of the limiting functions will be established by analyzing the contractive properties of this recursive system in appropriate function spaces.  

\begin{remark}[Preservation of nonnegativity]
\rm 
For any $F \in \mathcal{C}^1(K_{R^*,T^*})$, the iterative scheme \eqref{2.1}--\eqref{2.2} preserves non-negativity when:
\begin{itemize}
    \item[$(i)$] $F(x,0) \geq 0$ for all $x \in B_{R^* + T^*}$.
    \item[$(ii)$] Either $D_-F(x,t) \geq 0$ or $D_+F(x,t) \geq 0$ throughout $K_{R^*,T^*}$.
\end{itemize}
This property will be crucial for establishing the monotonicity of our approximate solutions.
\end{remark}

We now establish two fundamental technical lemmas that underpin our existence proof.

\begin{lemma}[Monotonicity of Approximations]\label{lem:2.1}
Under assumption $(A_2)$, the successive approximations satisfy:
$$
\phi_{n+1}(x,t) \geq \phi_n(x,t) \geq 0 \quad \text{and} \quad \psi_{n+1}(x,t) \geq \psi_n(x,t) \geq 0,
$$
for all $(x,t) \in K_{R^*,T^*}$ and every $n \in \mathbb{N} \cup \{0\}$. Moreover, this monotonicity holds uniformly across the domain.
\end{lemma}

\begin{proof}
We proceed by strong induction on $n$. 

\noindent
We start by proving the case $n=0$. From assumption $(A_2)$ and the integral form \eqref{2.2}, we estimate:
$$
\phi_1(x,t) \geq f(x+t) + \int_0^t \left(2^{-p}|\gamma_1+\gamma_2|^p - \mu\frac{\gamma_1+\gamma_2}{2}\right) ds \geq \gamma_1 = \phi_0(x,t) \geq 0,
$$
for all $(x,t)\in K_{R^*,T^*}$. Similarly, $\psi_1 \geq \psi_0 \geq 0$ in $K_{R^*,T^*}$.

Now, we proceed with the inductive step. Assume that $\phi_k \geq \phi_{k-1} \geq 0$ and $\psi_k \geq \psi_{k-1} \geq 0$ for all $1 \leq k \leq n$. The difference between successive iterates satisfies:
$$
\phi_{n+1} - \phi_n = \int_0^t \left[N(\phi_n,\psi_n) - N(\phi_{n-1},\psi_{n-1})\right](x+t-s,s) ds,
$$
where $\displaystyle N(\phi,\psi) = 2^{-p}|\phi+\psi|^p - \frac{\mu}{1+s}\frac{\phi+\psi}{2}$.

 Consider the function
$$
\zeta(x) = \left(\frac{x}{2}\right)^p - \frac{\mu}{1+t}\frac{x}{2}, \quad t > 0.
$$
For $x > \alpha := 2(\mu/p)^{1/(p-1)}$ (which holds since $\gamma_1+\gamma_2 > 2\Big(\tfrac{4}{3}\,\mu\Big)^{\!\frac{1}{p-1}}$), $\zeta$ is strictly increasing. By the induction hypothesis:
$$
\phi_n + \psi_n \geq \phi_{n-1} + \psi_{n-1} > \gamma_1 + \gamma_2 > \alpha,
$$
and thus:
$$
N(\phi_n,\psi_n) = \zeta(\phi_n+\psi_n) \geq \zeta(\phi_{n-1}+\psi_{n-1}) = N(\phi_{n-1},\psi_{n-1}).
$$

Finally, the integrand in $\phi_{n+1}-\phi_n$ is nonnegative, so $\phi_{n+1} \geq \phi_n \geq 0$. An identical argument shows that $\psi_{n+1} \geq \psi_n \geq 0$, completing thus the induction.
\end{proof}

\begin{lemma}[Characteristic Derivative Bounds]\label{em}
Under the assumptions $(A_2)$-$(A_4)$, the approximate solutions satisfy the following gradient estimates:
$$
\partial_t \phi_n(x,t) \geq (1+\varepsilon_0)|\partial_x \phi_n(x,t)| \quad \text{and} \quad \partial_t \psi_n(x,t) \geq (1+\varepsilon_0)|\partial_x \psi_n(x,t)|,
$$
uniformly for all $(x,t) \in K_{R^*,T^*}$ and every iteration $n \in \mathbb{N} \cup \{0\}$. These inequalities persist throughout the approximation scheme and reflect the dominant behavior of temporal derivatives over spatial derivatives in our system.
\end{lemma}
\begin{proof}
We proceed by carefully analyzing the evolution of the derivatives. Set $\lambda = 1 + \varepsilon_0$ and define for each $n \in \mathbb{N} \cup \{0\}$:
$$
\begin{aligned}
J_n &= \partial_t \phi_n + \lambda \partial_x \phi_n, & \tilde{J}_n &= \partial_t \phi_n - \lambda \partial_x \phi_n, \\
L_n &= \partial_t \psi_n + \lambda \partial_x \psi_n, & \tilde{L}_n &= \partial_t \psi_n - \lambda \partial_x \psi_n.
\end{aligned}
$$
Note that for $n=0$, we have trivial initial conditions: $J_0 = \tilde{J}_0 = L_0 = \tilde{L}_0 = 0$ on $K_{R^*,T^*}$.

Let us start with the case $n=0$. For the first iteration, we compute at $t=0$:
$$
J_1(x,0) = \partial_t \phi_1(x,0) + \lambda \partial_x \phi_1(x,0) = (2+\varepsilon_0)f^\prime(x) + 2^{-p}(\gamma_1+\gamma_2)^p - \frac{\mu}{2}(\gamma_1+\gamma_2),
$$
which is nonnegative by assumption $(A_4)$. The evolution satisfies:
$$
D_- J_1 = \partial_t D_- \phi_1 + \lambda \partial_x D_- \phi_1 = \frac{\mu}{2}(\gamma_1+\gamma_2)\frac{1}{(1+t)^2} \geq 0.
$$
Hence $J_1 \geq 0$ throughout $K_{R^*,T^*}$. Similar computations show $\tilde{J}_1, L_1, \tilde{L}_1 \geq 0$.

Now, assume that $J_n, L_n \geq 0$ for some $n \geq 1$. At $t=0$:
$$\begin{aligned}
J_{n+1}(x,0) &= \partial_t \phi_{n+1}(x,0) + \lambda \partial_x \phi_{n+1}(x,0) \\ &= (2+\varepsilon_0)f^\prime(x) + 2^{-p}(f(x)+g(x))^p - \frac{\mu}{2}(f(x)+g(x)) \geq 0,
\end{aligned}
$$
by $(A_4)$. For $t>0$, we compute:
$$
D_- J_{n+1} = \partial_t(\partial_t \phi_{n+1} + \lambda \partial_x \phi_{n+1}) - \partial_x(\partial_t \phi_{n+1} + \lambda \partial_x \phi_{n+1}),
$$
which expands to:
$$
D_- J_{n+1} = p2^{-p}(\partial_t \phi_n + \partial_t \psi_n)(\phi_n + \psi_n)^{p-1} - \frac{\mu}{1+t}\frac{\partial_t \phi_n + \partial_t \psi_n}{2} + \frac{\mu}{(1+t)^2}\frac{\phi_n + \psi_n}{2} 
$$
$$
- \lambda p2^{-p}(\phi_n + \psi_n)^{p-1}(\partial_x \phi_n + \partial_x \psi_n) + \frac{\mu}{1+t}\frac{\partial_x \phi_n + \partial_x \psi_n}{2}
$$
$$
= \left(\frac{p}{2}\right)^p (\phi_n + \psi_n)^{p-1}(J_n + L_n) - \frac{\mu}{1+t}(J_n + L_n) + \frac{\mu}{(1+t)^2}\frac{\phi_n + \psi_n}{2} \geq 0,
$$
since $\gamma_1 + \gamma_2 \geq \max(\mu^{1/(p-1)}(2/p)^{p/(p-1)},A)$, where $A$ is given by Lemma \ref{lem21}. Thus $J_{n+1} \geq 0$. Parallel arguments hold for $L_{n+1}, \tilde{J}_{n+1}, \tilde{L}_{n+1}$.

\medskip
\noindent
By induction, $J_n, \tilde{J}_n, L_n, \tilde{L}_n \geq 0$ for all $n \in \mathbb{N}$ on $K_{R^*,T^*}$, proving the lemma.
\end{proof}

\subsection*{Proof of Existence and Regularity of $\phi$ and $\psi$.}

Let $(x,t) \in K_{R^*, T^*}$. Since $\{\phi_n(x,t)\}_{n=0}^\infty$ and $\{\psi_n(x,t)\}_{n=0}^\infty$ are increasing sequences in $n$, we have
\begin{equation}
\left\{
  \begin{aligned}
    \lim_{n \to \infty} \phi_n(x,t) &= \sup_{n \in \mathbb{N}} \phi_n(x,t) = \phi(x,t), \\
    \lim_{n \to \infty} \psi_n(x,t) &= \sup_{n \in \mathbb{N}} \psi_n(x,t) = \psi(x,t).
  \end{aligned}
    \label{2.3}
    \right.
\end{equation}
From Lemma~\ref{lem:2.1}, it follows that $\phi$ and $\psi$ increase monotonically in $t$. Hence, there exists a function $T(x)$ such that
$$
  T(x) = \sup \{\, t \in (0, T^*) : (\phi + \psi)(x,t) < \infty \,\}, \qquad x \in B_{R^*}.
$$
We define
\begin{equation}
    \label{omega}
  \Omega = \{\, (x,t) : x \in B_{R^*}, \; 0 < t < T(x) \,\}.
\end{equation}

\begin{remark}
\emph{We will show in Section \ref{sec5} that $T$ is the blow-up curve for $\phi$ and $\psi$.}
\end{remark}

\begin{lemma}\label{lemma 2.3}
  Under the assumptions $(A_2)$--$(A_4)$, the pair $(\phi, \psi)$,  defined by \eqref{2.3}, is the unique solution of \eqref{1.5} in the class $(\mathcal{C}^{3,1}(\Omega))^2$.
\end{lemma}

\begin{proof}
  Define the set
  $$
    B(t) = \bigl\{ x \in B_{R^*+T^*}: \lvert x - \tilde{x} \rvert \le \tilde{t}-t \bigr\}, \quad (\tilde{x}, \tilde{t}) \in \Omega.
  $$

  \medskip
  \noindent\textbf{(Proof of regularity).}  
  We first show that \((\phi, \psi)\) is a \(\mathcal{C}^{3,1}(\overline{\Omega})^2\) solution of \eqref{1.5}.  
  The proof is carried out in two steps.

  \smallskip
  \noindent\textbf{Step 1.}  
  Fix \((x, \tilde{t}\,) \in \Omega\).  
  We claim that there exists a constant \(M_0 > 0\) such that
  \begin{equation}
    \label{eq:2.4}
    \lVert \phi + \psi \rVert_{L^\infty(B(t))} \le M_0, \quad \text{for all } t \in [0, \tilde{t}\,].
  \end{equation}

  \noindent We argue by contradiction. Define  
  $$
    Y_{\tilde{x}} = \{ x \in B_{R^*} : \lvert x - \tilde{x} \rvert \le \tilde{t} - T(x) \},
  $$
  where \(m\) denotes the 1-dimensional Lebesgue measure.  
  Suppose \eqref{eq:2.4} fails, i.e., \(m(Y_{\tilde{x}}) > 0\).  
  By the monotonicity of \(\phi_n + \psi_n\) in \(t\), we have  
  $$
    (x, t) \notin \Omega \quad \text{if } x \in Y_{\tilde{x}} \text{ and } t = x + \tilde{t} - \tilde{x} \text{ or } t = -x + \tilde{t} + \tilde{x}.
  $$
  Hence, either \(m(Y_{\tilde{x},+}) > 0\) or \(m(Y_{\tilde{x},-}) > 0\), where  
  $$
    Y_{\tilde{x},+} = \{ s \in (0, \tilde{t}) : s = x + \tilde{t} - \tilde{x}, \ x \in Y_{\tilde{x}} \},  
  $$
  $$
    Y_{\tilde{x},-} = \{ s \in (0, \tilde{t}) : s = -x + \tilde{t} + \tilde{x}, \ x \in Y_{\tilde{x}} \}.
  $$
  However, evaluating the iteration yields  
  $$
    \infty > (\phi_{n+1} + \psi_{n+1})(\tilde{x}, \tilde{t}) 
    \geq \int_{Y_{\tilde{x},+}} \left( 2^{-p} \lvert \phi_n + \psi_n \rvert^p - \frac{\mu}{1+t} \frac{\phi_n + \psi_n}{2} \right)(x + \tilde{t} - s, s) \, ds  
  $$
  $$
    + 2^{-p} \int_{Y_{\tilde{x},-}} \left( 2^{-p} \lvert \phi_n + \psi_n \rvert^p - \frac{\mu}{1+t} \frac{\phi_n + \psi_n}{2} \right)(-x + \tilde{t} + \tilde{x}, s) \, ds \to \infty \quad \text{as } n \to \infty,
  $$
  which is a contradiction. Thus, \eqref{eq:2.4} holds.

   \noindent\textbf{Step 2.} 
We now prove that $(\phi, \psi)$ is a $\mathcal{C}^{3,1}(\overline{\Omega})^2$ solution of \eqref{1.5}.  

Fix $(\tilde{x}, \tilde{t}) \in \Omega$ and set $K = K(x, \tilde{t})$.  
By localizing the problem, it suffices to show that $(\phi, \psi)$ is a $\mathcal{C}^{3,1}(\overline{K})^2$ solution of \eqref{1.5}.  

First, we establish the uniform convergence of $\phi_n$ to $\phi$ and $\psi_n$ to $\psi$ in $K = K(x, \tilde{t})$.  

By Step~1, there exists a constant $C_0 = C_0(x, \tilde{t}) > 0$ such that
\begin{equation}
  \label{eq:2.51}
  \|\phi_n + \psi_n\|_{L^\infty(B(t))} \leq C_0 \quad \text{for all } t \in [0, \tilde{t}\,] \text{ and } n \in \mathbb{N}.
\end{equation}

For $t \in [0, \tilde{t}\,]$ and $n \in \mathbb{N}$, we estimate the differences:
\begin{align*}
&\|\phi_{n+1} - \phi_n\|_{L^\infty(B(t))} + \|\psi_{n+1} - \psi_n\|_{L^\infty(B(t))} \\
&\quad \leq (p C_0^{p-1} + \mu) \int_0^t \left( 
\|\phi_n - \phi_{n-1}\|_{L^\infty(B(s_1))} + \|\psi_n - \psi_{n-1}\|_{L^\infty(B(s_1))} 
\right) ds_1 \\
&\quad \leq (p C_0^{p-1} + \mu)^2 \int_0^t \int_0^{s_1} \left( 
\|\phi_{n-1} - \phi_{n-2}\|_{L^\infty(B(s_2))} + \|\psi_{n-1} - \psi_{n-2}\|_{L^\infty(B(s_2))} 
\right) ds_2 ds_1.
\end{align*}

Iterating this estimate yields
$$
\|\phi_{n+1} - \phi_n\|_{L^\infty(B(t))} + \|\psi_{n+1} - \psi_n\|_{L^\infty(B(t))} 
\leq 4 C_0 \frac{\big((p C_0^{p-1} + \mu) T\big)^n}{n!} \to 0 \quad \text{as } n \to \infty.
$$

From \eqref{2.3}, we conclude the uniform convergence:
$$
\|\phi_n - \phi\|_{L^\infty(K_{x, \tilde{t}})} + \|\psi_n - \psi\|_{L^\infty(K_{x, \tilde{t}})} \to 0 \quad \text{as } n \to \infty.
$$

We now establish that $\phi, \psi \in W^{1,\infty}(K, (\tilde{x}, \tilde{t}))$. 

The directional derivatives satisfy the following recursive relations for $n \in \mathbb{N} \cup \{0\}$:
$$
\begin{cases}
D_- D_\theta \phi_{n+1} = p 2^{-p} (\phi_n + \psi_n)^{p-1} (D_\theta \phi_n + D_\theta \psi_n), \\
D_+ D_\theta \psi_{n+1} = p 2^{-p} (\phi_n + \psi_n)^{p-1} (D_\theta \phi_n + D_\theta \psi_n),
\end{cases}
$$
where $D_\theta v = (\cos \theta \, \partial_x + \sin \theta \, \partial_t)v$ is the directional derivative.

The initial conditions are given by:
$$
\begin{cases}
D_\theta \phi_{n+1}(x,0) = 
\begin{cases}
(\cos\theta+\sin\theta)\,f^\prime(x) + \sin\theta\cdot 2^{-p}(\gamma_1+\gamma_2)^p & (n = 0), \\
(\cos\theta+\sin\theta)\,f^\prime(x) + \sin\theta\cdot 2^{-p}(f+g)^p(x) & (n \geq 1),
\end{cases} \\[6pt]
D_\theta \psi_{n+1}(x,0) = 
\begin{cases}
(\cos\theta-\sin\theta)\,g^\prime(x) + \sin\theta\cdot 2^{-p}(\gamma_1+\gamma_2)^p & (n = 0), \\
(\cos\theta-\sin\theta)\,g^\prime(x) + \sin\theta\cdot 2^{-p}(f+g)^p(x) & (n \geq 1).
\end{cases}
\end{cases}
$$
\noindent Define the function
$$
W(t) = (C_0^p + \mu C_0) \exp\left((pC_0^{p-1} + \mu + 1) t\right),
$$
which satisfies the integral equation
\begin{equation}
W(t) = C_0^p + \mu C_0 + \int_0^t (pC_0^{p-1} + \mu + 1) W(s) \, ds.
\label{eq:2.4a}
\end{equation}

We prove by induction that for all $t \in [0, \tilde{t}\,]$ and $n \in \mathbb{N} \cup \{0\}$:
\begin{equation}
\|D_\theta \phi_n(\cdot, t)\|_{L^\infty(B(t))} < W(t), \quad 
\|D_\theta \psi_n(\cdot, t)\|_{L^\infty(B(t))} < W(t).
\label{eq:2.5}
\end{equation}

First, observe that $D_\theta \phi_0 = D_\theta \psi_0 = 0 \leq W(t)$ for all $t \geq 0$. Assuming the inequalities hold for $n$, we establish the bound for $n+1$, that is
\begin{equation}
\|2^{-p} (\phi_n + \psi_n)^{p-1} (\cdot, t) (D_\theta \phi_n + D_\theta \psi_n) (\cdot, t)\|_{L^\infty(B(t))} \leq C_0^{p-1} W(t).
\label{eq:2.6}
\end{equation}
From assumption $(A_4)$, we derive the key estimate:
\begin{equation}
\begin{aligned}
\|D_\theta \phi_{n+1}(\cdot, t)\|_{L^\infty(B(t))} 
&\leq 2 \|f^\prime\|_{L^\infty(B(0))} + 2^{-p} \|f + g\|_{L^\infty(B(0))}^p + \frac{\mu}{2} \|f + g\|_{L^\infty(B(0))} \\
&\quad + \int_0^t p \left\| 2^{-p} (\phi_n + \psi_n)^{p-1} (D_\theta \phi_n + D_\theta \psi_n) \right\|_{L^\infty(B(s))} ds \\
&\quad + \int_0^t \left\| \frac{-\mu}{2(1+s)} (D_\theta \phi_n + D_\theta \psi_n) \right\|_{L^\infty(B(s))} ds \\
&\quad + \int_0^t \left\| \frac{\sin \theta}{2} \cdot \frac{\mu}{(1+s)^2} (\phi_n + \psi_n) \right\|_{L^\infty(B(s))} ds \\
&\leq C_0^p + \mu C_0 + \int_0^t (p C_0^{p-1} + \mu + 1) W(s) \, ds = W(t),
\end{aligned}
\label{eq:2.7}
\end{equation}
for $t \in [0, \tilde{t}\,]$. The same bound holds for $\|D_\theta \psi_{n+1}(\cdot, t)\|_{L^\infty(B(t))}$.

\medskip

Define $C_1 = (C_0^p+\mu C_0) \exp((pC_0^{p-1}+\mu+1) T)$, yielding:
\begin{equation}
\|D_\theta \phi_n(\cdot, t)\|_{L^\infty(B(t))} \leq C_1, \quad \text{and} \quad \|D_\theta \psi_n(\cdot, t)\|_{L^\infty(B(t))} \leq C_1.
\label{eq:2.8}
\end{equation}
The estimate of the difference estimates shows that
\begin{equation}
\begin{aligned}
&\|D_\theta \phi_{n+1} - D_\theta \phi_n\|_{L^\infty(B(t))} + \|D_\theta \psi_{n+1} - D_\theta \psi_n\|_{L^\infty(B(t))} \\
&\leq \int_0^t (\mu + p C_0^{p-1}) \big(\|D_\theta \phi_n - D_\theta \phi_{n-1}\|_{L^\infty(B(s_1))} + \|D_\theta \psi_n - D_\theta \psi_{n-1}\|_{L^\infty(B(s_1))}\big) ds_1 \\
&\quad + \int_0^t \big(\mu + 2p(p - 1) C_1 C_0^{p-2}\big) \big(\|\phi_n - \phi_{n-1}\|_{L^\infty(B(s_1))} + \|\psi_n - \psi_{n-1}\|_{L^\infty(B(s_1))}\big) ds_1 \\
&\leq 4C_1 \frac{\big((\mu + p C_0^{p-1}) T\big)^n}{n!} + 4 C_0 C_2 T^n \big(p C_0^{p-1}\big)^{n-1} \sum_{j=1}^n \frac{1}{j! (n - j)!} \to 0,
\end{aligned}
\label{eq:2.91}
\end{equation}
where $C_2 = \mu+2p(p - 1)C_1C_0^{p-2}$.

\medskip

The uniform convergence implies the existence of limits:
$$
\big(\phi_\theta^{(1)}, \psi_\theta^{(1)}\big) \in \big(L^\infty(K_-(\tilde{x}, \tilde{t}))\big)^2,
$$
with
$$
\|D_\theta \phi_n - \phi_\theta^{(1)}\|_{L^\infty(K_-(\tilde{x}, \tilde{t}))} + \|D_\theta \psi_n - \psi_\theta^{(1)}\|_{L^\infty(K_-(\tilde{x}, \tilde{t}))} \to 0.
$$

This establishes $\big(\phi, \psi\big) \in \big(W^{1, \infty}(K_-(\tilde{x}, \tilde{t}))\big)^2$. Iterating the argument yields the higher regularity:
$$
\big(\phi, \psi\big) \in \big(W^{4, \infty}(K_-(\tilde{x}, \tilde{t}))\big)^2 = \big(\mathcal{C}^{3, 1}(K_-(\tilde{x}, \tilde{t}))\big)^2.
$$

\noindent\textbf{Proof of Uniqueness.}
We now establish the uniqueness of the solution $(\phi, \psi)$ to \eqref{1.5}. 
 
Suppose there exist two solutions $(\phi_1, \psi_1)$ and $(\phi_2, \psi_2)$ with corresponding blow-up curves $T_1$ and $T_2$. Define their domains:
$$
\Omega_j = \{(x, t) : x \in B_{R^*}, \ 0 < t < T_j(x)\} \quad \text{for } j = 1,2.
$$
For any $(\tilde{x}, \tilde{t}) \in \Omega_1 \cap \Omega_2$, we have $K_-(\tilde{x}, \tilde{t}) \subset \Omega_1 \cap \Omega_2$. Following the approach in Step 2, we obtain the estimate:
$$\begin{aligned}
&\sup_{0 \leq t' \leq \tilde{t}} \Big( \|\phi_1 - \phi_2\|_{L^\infty(B(t'))} + \|\psi_1 - \psi_2\|_{L^\infty(B(t'))} \Big)\\
&\leq (t p C_0^{p-1} + \mu) \sup_{0 \leq t' \leq \tilde{t}} \Big( \|\phi_1 - \phi_2\|_{L^\infty(B(t'))} + \|\psi_1 - \psi_2\|_{L^\infty(B(t'))} \Big).
\end{aligned}
$$
For sufficiently small $t > 0$, this implies:
$$
\sup_{0 \leq t' \leq t} \Big( \|\phi_1 - \phi_2\|_{L^\infty(B(t'))} + \|\psi_1 - \psi_2\|_{L^\infty(B(t'))} \Big) = 0.
$$
Since $C_0$ is time-independent, we can iterate this argument to extend the uniqueness to all $t' \in [0, \tilde{t}\,]$:
$$
\sup_{0 \leq t' \leq \tilde{t}} \Big( \|\phi_1 - \phi_2\|_{L^\infty(B(t'))} + \|\psi_1 - \psi_2\|_{L^\infty(B(t'))} \Big) = 0.
$$
This establishes that
 $(\phi_1, \psi_1) \equiv (\phi_2, \psi_2)$ in $\Omega_1 \cap \Omega_2$ and
 $T_1(x) = T_2(x)$ for all $x \in B_{R^*}$,
completing the proof of uniqueness and hence the lemma.
\end{proof}
\begin{lemma}
Under the assumptions A\ref{A1}--A\ref{A4}, the blow-up time satisfies
$$
T(x) < T^* \quad \text{for all } x \in B_{R^*},
$$
where $T(x)$ and $T^*$ are two positive real numbers as defined in \eqref{eq:blowup_time}.
\end{lemma}

\begin{proof}
We establish the result through a comparison argument with auxiliary functions.

Define sequences $\{\tilde{\phi}_n\}$ and $\{\tilde{\psi}_n\}$ by $\tilde{\phi}_0 = \gamma_1$, $\tilde{\psi}_0 = \gamma_2$ and for $n \geq 0$:
\begin{equation}
\begin{cases}
\dfrac{d}{dt}\tilde{\phi}_{n+1} = 2^{-p} |\tilde{\phi}_n + \tilde{\psi}_n|^p - \dfrac{\mu}{2}(\tilde{\phi}_n + \tilde{\psi}_n), & t > 0, \\[6pt]
\dfrac{d}{dt}\tilde{\psi}_{n+1} = 2^{-p} |\tilde{\phi}_n + \tilde{\psi}_n|^p - \dfrac{\mu}{2}(\tilde{\phi}_n + \tilde{\psi}_n), & t > 0, \\[6pt]
\tilde{\phi}_{n+1}(0) = \gamma_1, \quad \tilde{\psi}_{n+1}(0) = \gamma_2.
\end{cases}
\label{eq:2.9}
\end{equation}

For $(x,t) \in K_{R^*,T^*}$, we have:
\begin{align*}
\phi_1(x,t) - \tilde{\phi}_1(t) 
&= f(x+t) - \gamma_1 \\
&\quad + \int_0^t \left[2^{-p}|\phi_0+\psi_0|^p - \frac{\mu}{2(1+s)}(\phi_0+\psi_0)\right](x+t-s,s)\,ds \\
&\quad - \int_0^t \left[2^{-p}|\tilde{\phi}_0+\tilde{\psi}_0|^p - \frac{\mu}{2}(\tilde{\phi}_0+\tilde{\psi}_0)\right]\,ds \\
&\geq f(x+t) - \gamma_1 \geq 0.
\end{align*}
Similarly, $\psi_1(x,t) \geq \tilde{\psi}_1(t)$ in $K_{R^*,T^*}$.

Now, we proceed with the inductive step. For that purpose, we assume that $\phi_n(x,t) \geq \tilde{\phi}_n(t)$ and $\psi_n(x,t) \geq \tilde{\psi}_n(t)$ hold in $K_{R^*,T^*}$. Then, we have
\begin{align*}
\phi_{n+1}(x,t) - \tilde{\phi}_{n+1}(t) 
&\geq f(x+t) - \gamma_1 \\
&\quad + \int_0^t \left[2^{-p}|\phi_n+\psi_n|^p - \frac{\mu}{2}(\phi_n+\psi_n)\right](x+t-s,s)\,ds \\
&\quad - \int_0^t \left[2^{-p}|\tilde{\phi}_n+\tilde{\psi}_n|^p - \frac{\mu}{2}(\tilde{\phi}_n+\tilde{\psi}_n)\right]\,ds \geq 0.
\end{align*}
The same inequality holds for $\psi_{n+1}$. By induction, we conclude that
$$
\phi_n(x,t) \geq \tilde{\phi}_n(t), \quad \psi_n(x,t) \geq \tilde{\psi}_n(t) \quad \text{in } K_{R^*,T^*} \text{ for all } n \in \mathbb{N}.
$$

Since the comparison functions $\tilde{\phi}_n$, $\tilde{\psi}_n$ blow up before $T^*$, the same must hold for $\phi_n$, $\psi_n$, proving that $T(x) < T^*$.

This ends the proof of the lemma.
\end{proof}

\section{Blow-up rates of solutions and Lipschitz continuity of \texorpdfstring{$T$}{T}}\label{sec5}

Our objective is to prove the Lipschitz continuity of the blow-up time $T$ on $B_{R^*}$. The proof relies on Theorem~\ref{thm3.3}, so we begin by proving this theorem.
\begin{proof}[Proof of Theorem \ref{thm3.3}]
We begin by establishing the lower and upper bounds for $\partial_t \phi$. 

First, for all $(x,t) \in K_{R^*,T^*}$, consider the time derivative evolution for the iterative scheme \eqref{2.1}:
\begin{equation}\label{eq:3.6}
\begin{aligned}
D_- \partial_t \phi_{n+1} &= \partial_t D_- \phi_{n+1} = \partial_t \left(2^{-p}|\phi_n + \psi_n|^p\right) \\
&= p2^{-p}(\phi_n + \psi_n)^{p-1}(\partial_t \phi_n + \partial_t \psi_n) \\
&\quad + \frac{\mu}{(1+t)^2}\frac{\phi_n + \psi_n}{2} - \frac{\mu}{1+t}\frac{\partial_t \phi_n + \partial_t \psi_n}{2}.
\end{aligned}
\end{equation}
From Lemma \ref{em}, we obtain the following key inequality:
\begin{equation}\label{eq:3.7}
\begin{aligned}
D_- &\left(2^{-p}(\phi_n + \psi_n)^p - \frac{\mu}{1+t}\frac{\phi_n + \psi_n}{2}\right) \\
&= \left(p2^{-p}(\phi_n + \psi_n)^{p-1} - \frac{\mu}{2(1+t)}\right) 
   (\partial_t \phi_n - \partial_x \phi_n + \partial_t \psi_n - \partial_x \psi_n) \\
&\quad + \frac{\mu}{(1+t)^2}\frac{\phi_n + \psi_n}{2} \\
&\leq p2^{-p+1}(\phi_n + \psi_n)^{p-1}(\partial_t \phi_n + \partial_t \psi_n) \\
&\quad + \frac{\mu}{(1+t)^2}\frac{\phi_n + \psi_n}{2} - \frac{\mu}{1+t}(\partial_t \phi_n + \partial_t \psi_n), \quad \forall \ (x,t) \in K_{R^*,T^*}.
\end{aligned}
\end{equation}

The combination of \eqref{eq:3.6} and \eqref{eq:3.7} reveals the fundamental relationship between the time derivatives and the nonlinear terms. This structure allows us to establish the desired bounds through a comparison principle for the iterative scheme, a uniform control of the lower order terms involving $\mu$, and then a careful tracking of the nonlinear interactions between $\phi_n$ and $\psi_n$. This will be subsequently developed in the proof.

We establish the key bounds for $\partial_t \phi_{n+1}$ through careful analysis of the quantity
$$
J_{\phi,n+1} := 2\partial_t\phi_{n+1} - \left(2^{-p}(\phi_n+\psi_n)^p - \frac{\mu}{1+t}\frac{\phi_n+\psi_n}{2}\right).
$$
From \eqref{eq:3.6}--\eqref{eq:3.7}, we derive the crucial inequality
\begin{equation}
\label{eq:3.8}
D_- J_{\phi,n+1} \geq 0, \quad \text{for } (x,t) \in K_{R^*,T^*}.
\end{equation}
At $t=0$, using assumption $(A_4)$, we get  
\begin{equation}
\label{eq:3.9}
\begin{aligned}
J_{\phi,n+1}(x,0) &\geq 2f^\prime(x) + 2^{-p}(\gamma_1 + \gamma_2)^p - \frac{\mu}{2}(\gamma_1 + \gamma_2) \\
&\geq 0, \quad \forall \ x \in B_{R^*}.
\end{aligned}
\end{equation}
This implies $J_{\phi,n} \geq 0$ throughout $K_{R^*,T^*}$ for all $n \geq 1$.\\
Applying Lemma \ref{em}, we obtain the upper estimate
$$
\partial_t\phi_{n+1} \leq \frac{1+\varepsilon_0}{\varepsilon_0} \left(2^{-p}(\phi_n+\psi_n)^p - \frac{\mu}{1+t}\frac{\phi_n+\psi_n}{2}\right).
$$
Combining this with the nonnegativity of $J_{\phi,n}$ yields the two-sided bound
\begin{equation}
\label{eq:3.10}
2^{-p-2}(\phi_n+\psi_n)^p  \leq \partial_t\phi_{n+1} \leq \frac{1+\varepsilon_0}{\varepsilon_0} 2^{-p}(\phi_{n+1}+\psi_{n+1})^p,
\quad (x,t) \in K_{R^*,T^*}.
\end{equation}
The condition $\gamma_1+\gamma_2 > (\mu\,p\,2^{p})^{\frac{1}{p-1}}
$ ensures that $$2^{-p}(\phi_n+\psi_n)^p - \frac{\mu}{1+t}\frac{\phi_n+\psi_n}{2}\geq 2^{-p-2}(\phi_n+\psi_n)^p.$$ 
The estimate \eqref{eq:3.10} establishes \eqref{aa} and the analogous result for $\psi$ follows similarly through identical arguments.
We now establish the precise blow-up rates for the solution.
From \eqref{eq:3.10}, we have the key inequality:
$$
\frac{\partial(\phi+\psi)}{\partial t}(x,t) \leq 2^{-p+1}(1+\varepsilon_0)\varepsilon_0^{-1}(\phi+\psi)^p(x,t), \quad (x,t)\in\Omega.
$$
For any fixed $x_0 \in B_{R^*}$, integrating from $\tau$ to $T(x_0)$ yields:
$$
\int_{(\phi+\psi)(x_0,\tau)}^{\infty} \frac{du}{u^p} \leq 2^{-p+1}(1+\varepsilon_0)\varepsilon_0^{-1} (T(x_0)-\tau).
$$
This integration gives the upper bound in \eqref{bb} with constant:
$$
C_2 = 2^{p-1}(p-1)^{-1/(p-1)}.
$$
The lower bound follows similarly from the inequality:
$$
\frac{\partial(\phi+\psi)}{\partial t} \geq 2^{-p-2}(\phi+\psi)^p.
$$
Integration provides the lower bound in \eqref{bb} with:
$$
C_1 = 2\bigl((p-1)\varepsilon_0^{-1}(1+\varepsilon_0)\bigr)^{-1/(p-1)}.
$$
The estimates \eqref{bb} establish \eqref{ee} with $q=1/(p-1)$. The derivative estimates \eqref{cc} and \eqref{dd} follow by combining these results with the previous bounds on $\partial_t\phi$ and $\partial_t\psi$. 
\end{proof}
Having established the blow-up rates of solutions, we now prove the Lipschitz continuity of the blow-up time $T(x)$.
\begin{proof}[Proof of Theorem \ref{thm3.2}]
Fix an arbitrary $\varepsilon > 0$. From the lower bound in \eqref{ee}, there exists a constant $C_1 = C_1(p,\varepsilon_0) > 0$ such that for all $x \in B_{R^*}$ and $t \in [T(x)-\varepsilon, T(x)]$ we have:
$$
C_1(T(x)-t)^{-q} \leq (\phi+\psi)(x,t).
$$
In particular, at $t = T(x) - \varepsilon$ this yields:
$$
(\phi+\psi)(x, T(x)-\varepsilon) \geq C_1 \varepsilon^{-q}.
$$
Choose $M \geq C_1 \varepsilon^{-q}$ and define for each $x \in B_{R^*}$ the time $E(x)$ as the solution to:
$$
(\phi+\psi)(x, E(x)) = M, \quad \text{with } \ T(x)-\varepsilon \leq E(x) <  T(x).
$$
The existence of such $E(x)$ follows from the continuity of $(\phi+\psi)(x,\cdot)$, the blow-up property $\lim_{t\to T(x)^-} (\phi+\psi)(x,t) = +\infty$ and the intermediate value theorem.\\
Building upon the existence of the level set function $E(x)$ established previously, we now analyze its regularity properties to prove the Lipschitz continuity of $T(x)$.

\noindent\textbf{$(i)$ Continuity of the Level Set Function $E$.}
Let $\{x_n\}_n$ be a sequence converging to $x' \in B_{R^*}$, and set $t_n = E(x_n)$. For any convergent subsequence $t_{n_k} \to \eta$, the continuity of $\phi + \psi$ implies:
$$
(\phi + \psi)(x', \eta) = \lim_{k \to \infty} (\phi + \psi)(x_{n_k}, t_{n_k}) = M.
$$
Since $\partial_t(\phi + \psi) > 0$ in $\Omega$, the level set is unique, forcing $\eta = E(x')$. Thus $E$ is continuous at $x'$.

\medskip
\noindent\textbf{$(ii)$ Local Lipschitz Property of $E$.}
Fix $x' \in B_{R^*}$ and choose $h'>0$ such that $B(x',h') \subset \Omega$. By continuity, there exists $0<h''\le h'$ for which $(x,E(x))\in B(x',h'')$ for all $x\in (x'-h'',x'+h'')$. For any $x_1,x_2\in (x'-h'',x'+h'')$, define the auxiliary function
\(
H(\xi) := (\phi + \psi)\big(x_1 + \xi(x_2 - x_1), E(x_1) + \xi k\big), 
\) where $0 \leq \xi \leq 1$ and $ k := E(x_2) - E(x_1).$
Since $H(0) = H(1) = M$, Rolle's Theorem yields $\xi' \in (0,1)$ satisfying:
\begin{equation}
\label{eq:3.16}
(x_2 - x_1)\,\partial_x(\phi + \psi)(x_{\xi'}, t_{\xi'}) + k\,\partial_t(\phi + \psi)(x_{\xi'}, t_{\xi'}) = 0,
\end{equation}
where $(x_{\xi'}, t_{\xi'}) = (x_1 + \xi'(x_2 - x_1), E(x_1) + \xi' k)$. From Lemma \ref{em} and \eqref{eq:3.16}, we obtain:
$$
|E(x_2) - E(x_1)| = \left|\frac{\partial_x(\phi + \psi)}{\partial_t(\phi + \psi)}(x_{\xi'}, t_{\xi'})\right|\,|x_2 - x_1| \leq \frac{1}{1 + \varepsilon_0} |x_2 - x_1|.
$$

\medskip
\noindent\textbf{$(iii)$ Differentiability of $E$.}
The difference quotient satisfies:
$$
\frac{E(x + h) - E(x)}{h} = -\frac{\partial_x(\phi + \psi)}{\partial_t(\phi + \psi)}\big(x + \xi h, E(x) + \xi(E(x + h) - E(x))\big).
$$
By continuity of $\partial_x(\varphi + \psi)$, $\partial_t(\varphi + \psi)$ and $E$, we see that 
\(
E \in \mathcal{C}^{1}(B_{R^*})
\)
 with:
$$
E'(x) = -\frac{\partial_x(\phi + \psi)}{\partial_t(\phi + \psi)}(x, E(x)), \qquad x\in B_{R^*}.
$$
The mean value theorem then gives the global Lipschitz bound:
\begin{equation}
\label{eq:3.17}
|E(x') - E(x'')| \leq \frac{1}{1 + \varepsilon_0} |x' - x''| \quad \forall \ x', x'' \in B_{R^*}.
\end{equation}

\medskip
\noindent\textbf{$(iv)$ Lipschitz Continuity of $T$.}
For any $x', x'' \in B_{R^*}$ and $\varepsilon > 0$, we have
$$
|T(x') - T(x'')| \leq |T(x') - E(x')| + |E(x') - E(x'')| + |E(x'') - T(x'')| \leq 2\varepsilon + \frac{1}{1 + \varepsilon_0} |x' - x''|.
$$
Taking $\varepsilon \to 0^+$ establishes the desired inequality \eqref{eq:3.15}.
\end{proof}


To better understand the behavior near the blow-up surface, we first introduce a distance metric:

\begin{definition}\label{def:3.3}
By $d(x,t)$, we denote the distance from a point $(x,t)\in\Omega$ to $\Gamma=\{(x,T(x)):\,x\in B_{R^*}\}$.
\end{definition}

Having defined the distance function, we can now relate it to the time-to-blow-up:

\begin{remark}\label{rem:3.4}
\rm It follows from Lemma \ref{lemma 2.3} that
$$
\frac{T(x)-t}{\sqrt{2}}\;\leq\;d(x,t)\;\leq\;T(x)-t.
$$
By replacing $T(x)-t$ by $d(x,t)$ in Theorem \ref{thm3.3} we obtain the following corollary.
\end{remark}

This relationship between distance and time allows us to state more precise blow-up estimates:

\begin{corollary}\label{cor:3.5}
Assume that A\ref{A1}--A\ref{A4} hold. Then, there exist positive constants $C_1$ and $C_2$, depending only on $p$ and $\varepsilon_0$, such that
\begin{align}
C_1d^{-q}(x,t) &\le (\phi+\psi)(x,t) \le C_2d^{-q}(x,t), \label{eq:3.18}\\
C_1d^{-q-1}(x,t) &\le \partial_t\phi(x,t) \le C_2d^{-q-1}(x,t), \label{eq:3.19}\\
C_1d^{-q-1}(x,t) &\le \partial_t\psi(x,t) \le C_2d^{-q-1}(x,t), \label{eq:3.20}
\end{align}
where $q=1/(p-1)$.
\end{corollary}

We can further refine these estimates to examine each component of the solution separately:

\begin{remark}\label{lem:3.6}
\rm  Under assumptions A\ref{A1}--A\ref{A4}, there exist positive constants $C_1$ and $C_2$ (depending on $p,\mu $ and $\varepsilon_0$) such that for $q=1/(p-1)$, the following bounds hold uniformly in $\Omega$:
\begin{align}
C_1\bigl(T(x)-t\bigr)^{-q} &\le \phi(x,t) \le C_2\bigl(T(x)-t\bigr)^{-q}, \label{eq:3.21R}\\
C_1\bigl(T(x)-t\bigr)^{-q} &\le \psi(x,t) \le C_2\bigl(T(x)-t\bigr)^{-q}. \label{eq:3.22R}
\end{align}
In fact, we have the following estimates for $\phi(x,t)$ (the proof for $\psi(x,t)$ is analogous):  
By Corollary~\ref{cor:3.5} and Lemma~\ref{lemma 2.3}, there exist constants $c_1, c_2 > 0$ such that  
$$
\phi\bigl(x,T(x)-\varepsilon\bigr) \ge \int_{T(x)-2\varepsilon}^{T(x)-\varepsilon} 2^{-p-1}(\phi+\psi)^p \, ds \ge c_2 \varepsilon^{-q},
$$  
where the first inequality follows from the integral expression for $\phi$ and the second from the singularity analysis near $T(x)$. The upper bound $\phi(x,T(x)-\varepsilon) \leq C_2 \varepsilon^{-q}$ is immediate from Theorem~\ref{thm3.3}.
\end{remark}

\section{Linearity of the Blow-up Curve for Limit Solutions}\label{sec7}
In this section, we prepare the analytical tools and estimates that will be instrumental in establishing the $\mathcal{C}^1$-regularity of the blow-up curve  on $B_{R^*}$. The actual proof of this regularity result will be carried out in Section~\ref{section8}.
We will also establish Theorem \ref{th:3.6}, proving that the limiting blow-up curve is linear.

\subsection{Blow-up Limits and Regularity of the Blow-up Time}\label{sec6}
The key to our analysis lies in understanding the directional derivatives of the solutions. We work with the directional derivative operator:
$$
D_\theta = \sin\theta\,\partial_t + \cos\theta\,\partial_x, \quad \text{for } 0 \leq \theta < 2\pi.
$$
The following fundamental estimates control the growth of derivatives as we approach the blow-up surface:

\begin{lemma}\label{lem:4.1}
Under the assumptions A\ref{A1}--A\ref{A5}, there exist positive constants $\lambda_\alpha$ and $\lambda_\alpha^{*}$ (depending only on $p$, $\varepsilon_1$, and $\mu$) such that for all $(x,t) \in \Omega$ and $\alpha = 0,1,2,3$, we have:
\begin{equation}\label{eq:4.1}
\begin{aligned}
\max\bigl\{|D_\theta^{\alpha}\phi(x,t)|, |D_\theta^{\alpha}\psi(x,t)|\bigr\}
  &\leq \lambda_\alpha\bigl(\phi+\psi\bigr)^{p+(\alpha-1)/q}(x,t),  \\
  &\leq \lambda_\alpha^{*}\,d(x,t)^{-(pq+\alpha-1)},
\end{aligned}
\end{equation}
where $q=1/(p-1)$ represents the standard blow-up exponent.
\end{lemma}
\begin{proof}
The estimate \eqref{eq:4.1}$_2$ follows immediately from Corollary \ref{cor:3.5} once \eqref{eq:4.1}$_1$ is established. We therefore focus on proving \eqref{eq:4.1}$_1$.\\
From assumption A\ref{A3}, we obtain a uniform bound on the second derivatives:
\begin{equation}\label{eq:4.3}
\max_{x\in B_{R^{*}+T^{*}}} \bigl\{|f^{\prime\prime}(x)| + |g^{\prime\prime}(x)|\bigr\}  \leq\kappa(f+g)^{2p-1}(x),
\end{equation}
for some constant $\kappa > 0$.\\
The estimates for $\alpha=0$ and $\alpha=1$ follow directly from Lemma \ref{em} and Theorem \ref{thm3.3}.

\medskip
\noindent For $\alpha=2$, we first prove the following estimate by induction on $n$:
\begin{equation}\label{eq:4.4}
\begin{aligned}
\max\bigl\{ |D_\theta^{2}\phi_n(x,t)|, |D_\theta^{2}\psi_n(x,t)| \bigr\} 
&\leq \lambda_2(\phi_n + \psi_n)^{2p-1}(x,t), \\
&\quad \forall n \in \mathbb{N}\cup\{0\}, (x,t) \in K_{R^{*}+T^{*}}.
\end{aligned}
\end{equation}
where $\lambda_2 = \lambda_2(p,\mu,\varepsilon_1) > 0$.

\noindent\textit{Case $n=0$.} 
Since $D_\theta\phi_0 = D_\theta\psi_0 = 0$ in $K_{R^{*}+T^{*}}$, the claim \eqref{eq:4.4} holds trivially.
\noindent\textit{Inductive Step.} 
Assume \eqref{eq:4.4} holds for some $n \geq 0$. Using the case $\alpha=1$, we infer that
\begin{equation}\label{eq:4.5}
\begin{aligned}
|D_-D_\theta^{2}\phi_{n+1}(x,t)|
&= \Bigl| D_\theta^{2} \Bigl(2^{-p}(\phi_n+\psi_n)^p - \frac{\mu}{1+t} \frac{\phi_n+\psi_n}{2}\Bigr)(x,t) \Bigr| \\[4pt]
&\leq \bigl(2^{-p+2} p(p-1) \lambda_1^{2} + 2^{-p+1} p\lambda_2 + \mu  + 2\mu \lambda_1 + \mu \lambda_2 \bigr) \\
&\quad \times (\phi_n + \psi_n)^{3p-2}(x,t),
\end{aligned}
\end{equation}
for all $(x,t) \in K_{R^{*}+T^{*}}$.\\
From Lemma \ref{em}, we derive the key inequality:
\begin{equation}\label{eq:4.6}
\begin{aligned}
&D_-\left[\lambda_2(\phi_{n+1}+\psi_{n+1})^{2p-1}\right](x,t) \\
&= \lambda_2(2p-1)(\phi_{n+1}+\psi_{n+1})^{2p-2}(x,t) \cdot D_-(\phi_{n+1}+\psi_{n+1})(x,t) \\
&\geq \lambda_2(2p-1)\!\left(1+\frac{\varepsilon_1}{2(1+\varepsilon_1)}\right)
2^{-p}\!\left(1-\frac{1}{2p}\right)\,
(\phi_n+\psi_n)^{3p-2},
\end{aligned}
\end{equation}
 in $K_{R^{*}+T^{*}}$.\\
Define the auxiliary quantity:
$$
M_n(x,t) = \lambda_2(\phi_n+\psi_n)^{2p-1}(x,t) - D_\theta^{2}\phi_n(x,t).
$$
From  \eqref{eq:4.3}, we obtain the following initial data lower bound:
\begin{equation}\label{eq:4.7}
M_{n+1}(x,0) \geq (\lambda_2 - \kappa)(f+g)^{2p-1}(x), \quad x\in B_{R^{*}+T^{*}}.
\end{equation}
Combining Lemma \ref{em} with \eqref{eq:4.5} and \eqref{eq:4.6} yields
\begin{equation}\label{eq:4.8}
\begin{aligned}
D_- M_{n+1}(x,t) \ge\;&
\Bigg[
\lambda_2\!\left(
2^{-p}(2p-1)\!\left(1+\frac{\varepsilon_1}{2(1+\varepsilon_1)}\right)\!\left(1-\frac{1}{2p}\right)
-2^{-p+1}p-\mu
\right) \\
&\qquad\quad
-\Big(2^{-p+2}p(p-1)\lambda_1^{2}+\mu+2\mu\lambda_1\Big)
\Bigg]\,(\phi_n+\psi_n)^{3p-2}(x,t).
\end{aligned}
\end{equation}
By assumption A\ref{A5}, we have
$$
2^{-p}(2p-1)\!\left(1+\frac{\varepsilon_1}{2(1+\varepsilon_1)}\right)\!\left(1-\frac{1}{2p}\right)
-2^{-p+1}p-\mu >0.
$$
We now select $\lambda_2$ sufficiently large such that
\begin{equation}\label{eq:4.C2-choice}
\lambda_2 \geq \max\left\{
\kappa,\,
\frac{2^{-p+2}p(p-1)\lambda_1^{2}+\mu+2\mu\lambda_1}{2^{-p}(2p-1)\!\left(1+\frac{\varepsilon_1}{2(1+\varepsilon_1)}\right)\!\left(1-\frac{1}{2p}\right)
-2^{-p+1}p-\mu }
\right\}.
\end{equation}

With the above choice for $\lambda_2$, the estimates \eqref{eq:4.7} and \eqref{eq:4.8} infer that $M_{n+1}\geq 0$ in $K_{R^{*}+T^{*}}$, and consequently $M_n\geq 0$ for all $n\in\mathbb{N}\cup\{0\}$. This establishes that
$$
\lambda_2(\phi_n+\psi_n)^{2p-1} \geq D_\theta^{2}\phi_n \quad \text{in } K_{R^{*}+T^{*}}, \quad n\in\mathbb{N}.
$$
\noindent Similarly, since the argument is symmetric in $\phi_n$ and $\psi_n$, we also obtain
\begin{equation*}
\lambda_2(\phi_n+\psi_n)^{2p-1} \ge D_{\theta}^{2}\psi_n \quad \text{in } K_{R^{*}+T^{*}}, \quad n\in\mathbb{N}.
\end{equation*}
After possibly enlarging $\lambda_2$, we obtain the full set of inequalities
$$
\begin{cases}
\lambda_2(\phi_n+\psi_n)^{2p-1} \geq |D_\theta^{2}\phi_n|, \\
\lambda_2(\phi_n+\psi_n)^{2p-1} \geq |D_\theta^{2}\psi_n|,
\end{cases}
\quad (x,t)\in K_{R^{*}+T^{*}},\; n\in\mathbb{N}.
$$
Thus \eqref{eq:4.4} holds true. 
Taking $n$ to infinity we get  \eqref{eq:4.1}$_1$.
The case $\alpha=3$ follows by an identical argument.
\end{proof}
\medskip
Now, to analyze the local behavior near a point \( x_0\in B_{R^{*}} \), we introduce the following scaled functions for \( \lambda>0 \) and \( q=1/(p-1) \):
\begin{align}
\phi_\lambda(y,s)&=\lambda^{q}\,\phi\bigl(x_0+\lambda y,\,T(x_0)+\lambda s\bigr), \label{eq:scaled-phi}\\
\psi_\lambda(y,s)&=\lambda^{q}\,\psi\bigl(x_0+\lambda y,\,T(x_0)+\lambda s\bigr). \label{eq:scaled-psi}
\end{align}
The scaled functions satisfy:
\begin{equation}
\begin{cases}
D_-\phi_\lambda
   =2^{-p}\bigl(\phi_\lambda+\psi_\lambda\bigr)^{p}
     -\dfrac{\mu}{1+t}\dfrac{\phi_\lambda+\psi_\lambda}{2},\\[6pt]
D_+\psi_\lambda
   =2^{-p}\bigl(\phi_\lambda+\psi_\lambda\bigr)^{p}
     -\dfrac{\mu}{1+t}\dfrac{\phi_\lambda+\psi_\lambda}{2},
\end{cases}
\label{fk}
\end{equation}
for \((y,s)\in\Omega_\lambda\), where
\(
\Omega_\lambda
  :=\bigl\{(y,s)\in\mathbb{R}^{2}:
           (x_0+\lambda y,\;T(x_0)+\lambda s)\in\Omega\bigr\}
\).
Furthermore, we introduce the distance \( d_\lambda(y,s) \) from \( (y,s) \) to the scaled blow-up surface
$$
\Gamma_\lambda
  :=\bigl\{(y,s)\in\mathbb{R}^{2}:s=T_\lambda(y)\bigr\},
$$
which plays a crucial role in our analysis. 
The following lemma establishes the relationship between the original and scaled blow-up curves.

\begin{lemma}\label{lem:4.2}
Let $\lambda>0$. The scaled blow-up time \( T_\lambda \) relates to the original \( T \) through:
\begin{equation}
T_\lambda(y)
  =\frac{T(x_0+\lambda y)-T(x_0)}{\lambda}. 
  \label{fkk}
\end{equation}
\end{lemma}
\begin{proof}
From Remark \ref{lem:3.6}, we have constants \( C_1,C_2>0 \) (depending on \( p \) and \( \varepsilon_1 \)) such that:
$$
\begin{aligned}
\lambda^{q}C_1\,
  \bigl[ T(x_0 + \lambda y) - (T(x_0) + \lambda s) \bigr]^{-q}
&\leq \lambda^{q}\phi\bigl(x_0 + \lambda y,\, T(x_0) + \lambda s\bigr) \\[4pt]
&\leq \lambda^{q}C_2\,
  \bigl[ T(x_0 + \lambda y) - (T(x_0) + \lambda s) \bigr]^{-q}.
\end{aligned}
$$

The key observation is the scaling identity:
\begin{equation}
\lambda^{q}
  \bigl[T(x_0+\lambda y)-(T(x_0)+\lambda s)\bigr]^{-q}
= \left(\frac{T(x_0+\lambda y)-T(x_0)}{\lambda}-s\right)^{-q}.
\end{equation}
As $s \to \frac{T(x_0+\lambda y)-T(x_0)}{\lambda}$, $\phi_\lambda(y,s)\to +\infty$. Hence, $T_\lambda(y)=\frac{T(x_0+\lambda y)-T(x_0)}{\lambda}$.
The same argument applies to \(\psi_\lambda\).\\
\end{proof}
\begin{lemma}[Uniform scaled estimates]\label{lem:scaled-estimates}
Fix $x_0\in B_{R^*}$ and set $q=\tfrac1{p-1}$. For $\lambda>0$, let $\Omega_\lambda=\{(y,s):\,s<T_\lambda(y)\}$ and $d_\lambda(y,s)$ be the Euclidean distance to $\{s=T_\lambda(y)\}$. Under \textup{A\ref{A1}}–\textup{A\ref{A5}}, there exist constants $C_1, C_2>0$ and for $\alpha=0,1,2,3$, constants $C_{3,\alpha},C_{4,\alpha}>0$ (depending on $p,\mu,\varepsilon_1$) such that, for all $(y,s)\in\Omega_\lambda$,
\begin{align}
&C_1(\phi_\lambda+\psi_\lambda)^p \le \partial_s\phi_\lambda \le C_2(\phi_\lambda+\psi_\lambda)^p, \label{4.14}\\
&C_1(\phi_\lambda+\psi_\lambda)^p \le \partial_s\psi_\lambda \le C_2(\phi_\lambda+\psi_\lambda)^p, \\
&C_1\bigl(T_\lambda(y)-s\bigr)^{-q} \le \phi_\lambda(y,s) \le C_2\bigl(T_\lambda(y)-s\bigr)^{-q}, \\
&C_1\bigl(T_\lambda(y)-s\bigr)^{-q} \le \psi_\lambda(y,s) \le C_2\bigl(T_\lambda(y)-s\bigr)^{-q}, \\
&|\partial_y\phi_\lambda| \le \frac{1}{1+\varepsilon_1}\,\partial_s\phi_\lambda,\quad
 |\partial_y\psi_\lambda| \le \frac{1}{1+\varepsilon_1}\,\partial_s\psi_\lambda, \\
&|T_\lambda(y)-T_\lambda(y')| \le \frac{1}{1+\varepsilon_1}\,|y-y'|, \label{20} \\
&\frac{T_\lambda(y)-s}{\sqrt{2}} \le d_\lambda(y,s) \le T_\lambda(y)-s, \\
&\max\{|D_\theta^{\alpha}\phi_\lambda|, |D_\theta^{\alpha}\psi_\lambda|\}
   \le C_{3,\alpha} (\phi_\lambda+\psi_\lambda)^{p+(\alpha-1)/q}
   \le C_{4,\alpha}\,d_\lambda^{-(pq+\alpha-1)}, \label{5.23}
\end{align}
where $D_\theta=\sin\theta\,\partial_s+\cos\theta\,\partial_y$, $\phi_\lambda(y,s)$ and $\psi_\lambda(y,s)$ are defined in \eqref{eq:scaled-phi} and \eqref{eq:scaled-psi}, $T_\lambda(y)=\big(T(x_0+\lambda y)-T(x_0)\big)/\lambda$  and all the above constants are independent of $\lambda$.
\end{lemma}
\begin{proof}
Fix $x_0\in B_{R^*}$ and set $q=\frac{1}{p-1}$. 
We use the scalings
$$
\partial_s\phi_\lambda=\lambda^{q+1}\partial_t\phi,\qquad 
\partial_y\phi_\lambda=\lambda^{q+1}\partial_x\phi,
$$
(and similarly for $\psi_\lambda$) and
$$
\phi_\lambda+\psi_\lambda=\lambda^{q}(\phi+\psi).
$$
By Theorem~\ref{thm3.3} and the identity $pq=q+1$, we have
$$
C_1(\phi+\psi)^p\le\partial_t\phi\le C_2(\phi+\psi)^p,
\qquad
C_1(\phi+\psi)^p\le\partial_t\psi\le C_2(\phi+\psi)^p.
$$
Multiplying by $\lambda^{q+1}$ and using the previous scaling yields
$$
C_1(\phi_\lambda+\psi_\lambda)^p\le\partial_s\phi_\lambda\le C_2(\phi_\lambda+\psi_\lambda)^p,
\qquad
C_1(\phi_\lambda+\psi_\lambda)^p\le\partial_s\psi_\lambda\le C_2(\phi_\lambda+\psi_\lambda)^p.
$$
Next, Remark~\ref{lem:3.6} gives
$$
C_1\,(T(x)-t)^{-q}\le \phi(x,t),\psi(x,t)\le C_2\,(T(x)-t)^{-q}.
$$
Since $T(x_0+\lambda y)-\big(T(x_0)+\lambda s\big)=\lambda\big(T_\lambda(y)-s\big)$, multiplying by $\lambda^{q}$ gives
$$
C_1\,(T_\lambda(y)-s)^{-q}\le \phi_\lambda(y,s),\psi_\lambda(y,s)\le C_2\,(T_\lambda(y)-s)^{-q}.
$$
From Lemma~\ref{em} we have $\partial_t\phi\ge(1+\varepsilon_1)|\partial_x\phi|$ and $\partial_t\psi\ge(1+\varepsilon_1)|\partial_x\psi|$; 
under the scaling this becomes
$$
|\partial_y\phi_\lambda|\le\frac{1}{1+\varepsilon_1}\,\partial_s\phi_\lambda,
\qquad
|\partial_y\psi_\lambda|\le\frac{1}{1+\varepsilon_1}\,\partial_s\psi_\lambda.
$$
The Lipschitz estimate of Theorem~\ref{thm3.2} implies 
$$
|T_\lambda(y)-T_\lambda(y')|\le L\,|y-y'|
\quad\text{with}\quad 
L:=(1+\varepsilon_1)^{-1}<1.
$$
Hence the Euclidean distance $d_\lambda(y,s)$ from $(y,s)$ to the graph $s=T_\lambda(y)$ satisfies the standard geometric bounds
$$
\frac{T_\lambda(y)-s}{\sqrt{1+L^2}}\le d_\lambda(y,s)\le T_\lambda(y)-s,
$$
and since $L<1$ we obtain
$$
\frac{T_\lambda(y)-s}{\sqrt{2}}\le d_\lambda(y,s)\le T_\lambda(y)-s.
$$
Finally, Lemma~\ref{lem:4.1} in original variables yields, for $\alpha=0,1,2,3$,
$$
\max\{|D_\theta^\alpha\phi|,|D_\theta^\alpha\psi|\}\le \lambda_\alpha\,(\phi+\psi)^{p+(\alpha-1)/q}.
$$
Under the scaling $D_\theta$ (a linear combination of $\partial_s$ and $\partial_y$) each derivative contributes a factor $\lambda^{q+1}$, while 
$(\phi+\psi)^{p+(\alpha-1)/q}$ contributes $\lambda^{pq+\alpha-1}=\lambda^{q+\alpha}$, so the inequality is scale-invariant and becomes
$$
\max\{|D_\theta^\alpha\phi_\lambda|,|D_\theta^\alpha\psi_\lambda|\}\le C_{3,\alpha}(\phi_\lambda+\psi_\lambda)^{p+(\alpha-1)/q}.
$$
Using $(\phi_\lambda+\psi_\lambda)\le C\,(T_\lambda-s)^{-q}$ from above and the comparison between $T_\lambda-s$ and $d_\lambda$ gives
$$
\max\{|D_\theta^\alpha\phi_\lambda|,|D_\theta^\alpha\psi_\lambda|\}\le C_{4,\alpha}\,d_\lambda^{-(pq+\alpha-1)}.
$$
All constants depend only on $(p,\mu,\varepsilon_1)$ and are independent of $\lambda$, completing thus the proof.
\end{proof}

\subsection{Linearity of the Blow-up Curve for Limit Solutions}\label{sec6.2}
Our approach to establishing \( T \in \mathcal{C}^1(B_{R^*}) \), which will be the subject of Section \ref{section8},    involves analyzing limits of the scaled functions \( T_{\lambda_n} \), \( \phi_{\lambda_n} \), and \( \psi_{\lambda_n} \).
For that purpose, let us introduce $\{\phi_{\lambda_n}\}$ and $\{\psi_{\lambda_n}\}$ with $\lambda_n \downarrow 0$ which  are called blow-up sequences (see \cite{Caffarelli1986}).

\medskip
Let \(\{I_k\}_{k\ge1}\) be an increasing family of closed intervals covering \(\mathbb{R}\).
By the Lipschitz estimate \eqref{20}, the family \(\{T_\lambda\}_{\lambda>0}\) is equicontinuous on each \(I_k\).
Moreover, there exist \(\lambda_k>0\) and \(M_k>0\) such that
$$
\sup_{0<\lambda<\lambda_k}\ \sup_{y\in I_k}\,|T_\lambda(y)|\ \le\ M_k.
$$
Hence, by Arzelà–Ascoli, for every \(k\) there exists a subsequence \((\lambda^{(k)}_n)_{n\ge 1}\) with
\(T_{\lambda_n^{(k)}}\to T_0^{(k)}\) in \(C(I_k)\).
A diagonal extraction \(\Lambda_n:=\lambda_n^{(n)}\) then yields
$$
T_{\Lambda_n}\ \longrightarrow\ T_0 \qquad\text{locally uniformly on }\mathbb{R}.
$$
We will  show that \(T_0\) is affine, namely \(T_0(y)=\alpha_{x_0}y\), and thus preclude any loss of differentiability for \(T\) at \(x_0\).
If \(T\) were not differentiable at some \(x'\), one could build sequences producing
\begin{equation}\label{eq:4.22}
\limsup_{\lambda_{n'}\to 0}\frac{T_{\lambda_{n'}}(y')}{y'}\;>\;
\liminf_{\lambda_{n''}\to 0}\frac{T_{\lambda_{n''}}(y'')}{y''},
\end{equation}
which would give two distinct affine limits \(T_0^{(1)}(y)=\alpha_{x_0}y\) and \(T_0^{(2)}(y)=\alpha'_{x_0}y\).
This contradiction rules out non-differentiability and ensures the differentiability of \(T\) at \(x_0\).

\medskip
We  examine the limiting behavior of our scaled solutions as \(\lambda \to 0\).
The natural domain for the limiting solutions is
$$
\Omega_0 := \{(y,s) : y \in \mathbb{R},\ s < T_0(y)\},
$$
where \(T_0\) is the uniform limit obtained previously.
\medskip
 From \eqref{20} and \eqref{5.23}, for every compact $K\Subset \Omega_0$ there exists $C=C(K)>0$ such that, for all small $\lambda$ with $K\subset \Omega_\lambda$,
\begin{equation}\tag{0.4}
C^{-1}\bigl(T_\lambda(y)-s\bigr)^{-q}\le \phi_\lambda,\ \psi_\lambda \le
C\,\bigl(T_\lambda(y)-s\bigr)^{-q}\quad \text{on }K,
\end{equation}
\begin{equation}\tag{0.5}
\max_{\theta\in[0,2\pi)}\ \max_{\alpha=0,1,2,3}
\Bigl(\,\bigl|D_\theta^{\alpha}\phi_\lambda\bigr|+\bigl|D_\theta^{\alpha}\psi_\lambda\bigr|\,\Bigr)\le C
\quad \text{on }K.
\end{equation}
Hence $\{\phi_\lambda\}$, $\{\psi_\lambda\}$ and their derivatives up to order $3$ are equibounded and equicontinuous on $K$. Applying Arzelà--Ascoli on an exhausting family of compacts $K\Subset \Omega_0$ and taking a diagonal subsequence of $\{\Lambda_n\}$, we obtain $\tilde\lambda_n\downarrow 0$ and functions $v_\phi, v_\psi$ such that, for every multi-index $\beta$ with $|\beta|\le 3$,
\begin{equation}\tag{0.6}
\partial^\beta \phi_{\tilde\lambda_n}\to \partial^\beta v_\phi,\qquad
\partial^\beta \psi_{\tilde\lambda_n}\to \partial^\beta v_\psi
\quad \text{uniformly on any compact subset of }\Omega_0.
\end{equation}
Equivalently, we have, as $n \to \infty$, 
\begin{equation}\label{eq:4.23}
\begin{cases}
\phi_{\tilde{\lambda}_n} \to v_\phi, \quad \psi_{\tilde{\lambda}_n} \to v_\psi, \\[4pt]
D_\theta\phi_{\tilde{\lambda}_n} \to v_\phi^{1,\theta}, \quad D_\theta\psi_{\tilde{\lambda}_n} \to v_\psi^{1,\theta}, \\[4pt]
D_\theta^2\phi_{\tilde{\lambda}_n} \to v_\phi^{2,\theta}, \quad D_\theta^2\psi_{\tilde{\lambda}_n} \to v_\psi^{2,\theta}, \\[4pt]
D_\theta^3\phi_{\tilde{\lambda}_n} \to v_\phi^{3,\theta}, \quad D_\theta^3\psi_{\tilde{\lambda}_n} \to v_\psi^{3,\theta}.
\end{cases}
\end{equation}
Since $\phi_{\tilde\lambda_n},\ \psi_{\tilde\lambda_n}\in \mathcal{C}^3$ and all derivatives $\partial^\beta$ with $|\beta|\le 3$ converge locally uniformly we have
$$
v_\phi,\ v_\psi \in \mathcal{C}^3(\Omega_0),\qquad
\partial^\beta v_\phi=\lim_{n\to\infty}\partial^\beta \phi_{\tilde\lambda_n},\quad
\partial^\beta v_\psi=\lim_{n\to\infty}\partial^\beta \psi_{\tilde\lambda_n}\quad (|\beta|\le 3).
$$
Thus the convergences hold in $\mathcal{C}^3_{\mathrm{loc}}(\Omega_0)$.
The profiles $v_\phi,\ v_\psi$ are limits of the rescaled functions $\phi_\lambda,\ \psi_\lambda$. Equivalently, along $\tilde\lambda_n\downarrow 0$,
\begin{equation}\tag{0.7}
\phi\bigl(x_0+\tilde\lambda_n y,\ T(x_0)+\tilde\lambda_n s\bigr)
=\tilde\lambda_n^{-q}\,v_\phi(y,s)+o\bigl(\tilde\lambda_n^{-q}\bigr),
\end{equation}
uniformly on compact subsets of $\Omega_0$; the same holds for $\psi$.
By  \eqref{fk} and \eqref{5.23} we get the following limiting system:
\begin{equation}\label{eq:4.24}
\begin{aligned}
D_-v_\phi &= 2^{-p}(v_\phi + v_\psi)^p, \\[4pt]
D_+v_\psi &= 2^{-p}(v_\phi + v_\psi)^p.
\end{aligned}
\end{equation}
The following estimates hold for \((y,s) \in \Omega_0\) with constants \(C_1, C_2, C_{3,\alpha}, C_{4,\alpha}\) depending only on \(p,\mu\) and \(\varepsilon_1\) and the fixed parameters of the construction:
\begin{align}
C_1(v_\phi + v_\psi)^p &\leq \partial_s v_\phi \leq C_2(v_\phi + v_\psi)^p, \label{eq:4.25} \\
C_1(v_\phi + v_\psi)^p &\leq \partial_s v_\psi \leq C_2(v_\phi + v_\psi)^p, \label{eq:4.26} \\
C_1(T_0(y) - s)^{-q} &\leq v_\phi(y,s) \leq C_2(T_0(y) - s)^{-q}, \label{eq:4.27} \\
C_1(T_0(y) - s)^{-q} &\leq v_\psi(y,s) \leq C_2(T_0(y) - s)^{-q}, \label{eq:4.28} \\
|\partial_y v_\phi| &\leq \frac{1}{1+\varepsilon_1}\,\partial_s v_\phi, \quad
|\partial_y v_\psi| \leq \frac{1}{1+\varepsilon_1}\,\partial_s v_\psi, \label{eq:4.29} \\
|T_0(y) - T_0(y_0)| &\leq \frac{|y - y_0|}{1 + \varepsilon_1}, \label{eq:4.30} \\
\frac{T_0(y) - s}{\sqrt{2}} &\leq d_0(y,s) \leq T_0(y) - s, \label{eq:4.31}
\end{align}
where \(d_0(y,s)\) is the distance to the limiting blow-up surface \(\Gamma_0 = \{(y,s) : s = T_0(y)\}\).
For derivatives up to third order (\(\alpha = 0,1,2,3\)):
\begin{equation}\label{eq:4.32}
\begin{aligned}
\max\{|D_\theta^\alpha v_\phi(y,s)|, |D_\theta^\alpha v_\psi(y,s)|\} 
&\leq C_{3,\alpha}\,(v_\phi + v_\psi)^{p+(\alpha-1)/q} \\
&\leq C_{4,\alpha}\,d_0(y,s)^{-(pq+\alpha-1)}.
\end{aligned}
\end{equation}

Now, we consider the limiting system
\begin{equation}\label{eq:5.1}
\left\{
\begin{aligned}
D_-V_\phi &= 2^{-p}(V_\phi + V_\psi)^p, \\
D_+V_\psi &= 2^{-p}(V_\phi + V_\psi)^p,
\end{aligned}
\right.
\end{equation}
with a straight blow-up curve of the form
\begin{equation}\label{eq:5.2}
\Gamma_\alpha := \{(y,s) : s = \alpha y, \ -\infty < y < \infty\},
\end{equation}
where \(\alpha \in \mathbb{R}\) is a constant.
\begin{lemma}
The system \eqref{eq:5.1} admits an explicit similarity solution
\begin{equation}\label{eq:5.3}
\left\{
\begin{aligned}
V_{\phi,\alpha}(y,s) &= C_{\phi,\alpha}\,(\alpha y - s)^{-q}, \\
V_{\psi,\alpha}(y,s) &= C_{\psi,\alpha}\,(\alpha y - s)^{-q},
\end{aligned}
\right.
\end{equation}
with constants given by:
$$
C_{\phi,\alpha}=\frac{1-\alpha}{2}\,A,\qquad
C_{\psi,\alpha}=\frac{1+\alpha}{2}\,A,\qquad
A^{\,p-1}=2^{\,p-1}\,q\,(1-\alpha^2).
$$
\end{lemma}
Before proving  Theorem~\ref{th:3.6},  we require several preparatory results.
\begin{lemma}\label{lem:5.2}
Under the assumptions A\ref{A1}--A\ref{A5}, the limiting blow-up time $T_0$ is concave.
\end{lemma}
\begin{proof}
Set $w:=v_\phi+v_\psi$. By \eqref{eq:4.25}--\eqref{eq:4.26} we get $\partial_s w>0$ and by \eqref{eq:4.29} we have 
$$
|\partial_y w|\le \frac{1}{1+\varepsilon_1}\,\partial_s w.
$$
In the coordinates
$$
u:=s+y,\quad v:=s-y,\qquad
\partial_u=\tfrac12(\partial_s+\partial_y),\ \partial_v=\tfrac12(\partial_s-\partial_y).
$$
This implies
\begin{equation}\label{100}
\partial_u w\ge0,\qquad \partial_v w\ge0.
\end{equation}
Next, starting from
$$
D_-v_\phi=2^{-p}w^p,\qquad D_+v_\psi=2^{-p}w^p,
$$
 computation gives
$$
(\partial_s^2-\partial_y^2)w
=2^{1-p}p\,w^{p-1}\,\partial_s w\ \ge\ 0.
$$
Hence,
\begin{equation}\label{200}
\partial_{uv}w=\tfrac14(\partial_s^2-\partial_y^2)w\ \ge\ 0.
\end{equation}
Let $M>0$ and define $\mathcal L_M:=\{(y,s): w(y,s)\le M\}$. Then, let $R=[u_1,u_2]\times[v_1,v_2]$ be any rectangle in the $(u,v)–$plane. Integrating \eqref{200} over $R$ yields
\begin{equation}\label{300}
w(u_2,v_2)+w(u_1,v_1)\ \ge\ w(u_2,v_1)+w(u_1,v_2).
\end{equation}
If the two opposite corners $(u_1,v_1)$ and $(u_2,v_2)$ lie in $\mathcal L_M$, then by \eqref{300} the other two corners also satisfy $\le M$, and by the monotonicities in \eqref{100} the whole rectangle $R$ lies in $\mathcal L_M$.\\ In particular, for any two points of $\mathcal L_M$, the segment between them stays in $\mathcal L_M$. Hence $\mathcal L_M$ is convex.\\
Finally, since $w(y,s)\to\infty$ as $s\nearrow T_0(y)$, we have
$$
\Omega_0=\{(y,s): s<T_0(y)\}=\bigcup_{M>0}\mathcal L_M,
$$
an increasing union of convex sets. Therefore, $\Omega_0$ is convex.\\ Because $\Omega_0$ is exactly the set $\{s<T_0(y)\}$, this gives that $T_0$ is concave.
\end{proof}
Now, we define the rescaled functions
$$
v_{\varphi,\lambda}(y,s) \;=\; \lambda^{q}\,v_{\varphi}(\lambda y,\lambda s), 
\qquad
v_{\psi,\lambda}(y,s) \;=\; \lambda^{q}\,v_{\psi}(\lambda y,\lambda s),
$$
with \(q=\frac{1}{p-1}\), for \(\lambda \to \infty\).
A direct calculation shows that the corresponding blow-up curve for these rescaled functions is given by
$$
T_{0,\lambda}(y) \;=\; \frac{T_0(\lambda y)}{\lambda}.
$$

\begin{lemma}\label{lemma:5.3}
Under the assumptions A\ref{A1}--A\ref{A5}, the rescaled blow-up curve converges to a piecewise linear function:
$$
T_{0,\lambda}(y) \;\longrightarrow\; 
\begin{cases}
\alpha\, y & \text{for } y \geq 0, \\[4pt]
\beta\, y  & \text{for } y < 0,
\end{cases}
\quad \text{as } \lambda \to \infty,
$$
where the constants $\alpha$ and $\beta$ satisfy $-1 < \alpha \leq \beta < 1$.
\end{lemma}

\begin{proof}
First, \(T_{0,\lambda}(0)=T_0(0)/\lambda=0\).
Fix \(y>0\) and define \(\phi(r):=T_0(r y)\) for \(r>0\).
By Lemma~\ref{lem:5.2}, \(T_0\) is concave; hence \(\phi\) is concave on \((0,\infty)\) with \(\phi(0)=0\).
Therefore the ratio \(r\mapsto \phi(r)/r\) is nonincreasing on \((0,\infty)\).\\
Let \(\{\lambda_n\}_{n\ge1}\) be any monotone increasing sequence with \(\lambda_n\to\infty\).
Then, we have
$$
\frac{T_{0,\lambda_n}(y)}{y}
=\frac{T_0(\lambda_n y)}{\lambda_n y}
=\frac{\phi(\lambda_n)}{\lambda_n}.
$$
is a nonincreasing sequence in \(n\).
Hence, the limit
$$
\alpha:=\lim_{n\to\infty}\frac{T_{0,\lambda_n}(y)}{y}
=\inf_{n}\frac{T_0(\lambda_n y)}{\lambda_n y}
=\inf_{\lambda>0}\frac{T_0(\lambda y)}{\lambda y},
$$
exists and is independent of both \(y>0\) and the choice of the increasing sequence \(\{\lambda_n\}\).
Consequently, we obtain
$$
T_{0,\lambda_n}(y)=\frac{T_0(\lambda_n y)}{\lambda_n}\ \longrightarrow\ \alpha\,y
\qquad\text{as }n\to\infty,\quad (y>0).
$$
For \(y<0\), write \(y=-|y|\) and define \(\psi(r):=T_0(-r|y|)\) for \(r>0\).
Again by concavity, \(r\mapsto \psi(r)/r\) is  nondecreasing on \((0,\infty)\).
With the same increasing sequence \(\{\lambda_n\}\),
$$
\frac{T_{0,\lambda_n}(y)}{y}
=\frac{T_0(-\lambda_n |y|)}{-\lambda_n |y|}
=\frac{\psi(\lambda_n)}{\lambda_n},
$$
is a nondecreasing sequence, so the limit
$$
\beta:=\lim_{n\to\infty}\frac{T_{0,\lambda_n}(y)}{y}
=\sup_{n}\frac{T_0(\lambda_n y)}{\lambda_n y}
=\sup_{\lambda>0}\frac{T_0(\lambda y)}{\lambda y},
$$
exists and is independent of both \(y<0\) and \(\{\lambda_n\}_n\).
Thus, we infer that
$$
T_{0,\lambda_n}(y)\ \longrightarrow\ \beta\,y
\qquad\text{as }n\to\infty,\quad (y<0).
$$
Finally, the Lipschitz bound \eqref{20} yields
\(
\bigl|\tfrac{T_{0,\lambda_n}(y)}{y}\bigr|\le \tfrac{1}{1+\varepsilon_1}
\)
for every \(n\).
Passing to the limit gives
\(|\alpha|,|\beta|\le \tfrac{1}{1+\varepsilon_1}<1\).
Therefore, we get
$$
T_{0,\lambda_n}(y)\to \alpha y \quad (y\ge0), 
\qquad
T_{0,\lambda_n}(y)\to \beta y \quad (y<0),
$$
with \(-1<\alpha\le \beta<1\).
\end{proof}
Now, We Consider the rescaled functions $v_{\phi,\lambda}$ and $v_{\psi,\lambda}$ in the domain
$$
\Omega_{0,\lambda} = \{(y,s) \in \mathbb{R}^2 : (\lambda y, \lambda s) \in \Omega_0\}.
$$
We Define the piecewise linear function and corresponding domain:
\begin{equation}\label{tild}
\tilde{T}_0(y) = 
\begin{cases}
\alpha y & \text{for } y \geq 0, \\
\beta y & \text{for } y < 0,
\end{cases}
\qquad
\tilde{\Omega}_0 = \{(y,s) \in \mathbb{R}^2 : s < \tilde{T}_0(y)\}.
\end{equation}
There exists a sequence $\{\lambda_n\}$ such that the following convergences hold locally uniformly in $\tilde{\Omega}_0$:
$$
v_{\phi,\lambda_n} \to w_\phi \quad \text{and} \quad v_{\psi,\lambda_n} \to w_\psi \quad \text{as } \lambda_n \to \infty,
$$
where $w_\phi, w_\psi \in \mathcal{C}^3(\tilde{\Omega}_0)$. 
\begin{remark}
Following the proof technique of Lemma \ref{lem:5.2}, we deduce that $\tilde{T}_0$ is concave. This implies that $\alpha$ and $\beta$ must satisfy either:
\begin{itemize}
\item $\alpha$ and $\beta$ have the same sign, or
\item $-1 < \alpha \leq 0 \leq \beta < 1$.
\end{itemize}
To establish Lemma \ref{th:3.6}, it suffices to prove:
\begin{equation}\label{7.7}
\alpha = \beta, \quad w_\phi = V_{\phi, \beta}, \quad w_\psi = V_{\psi, \beta}.
\end{equation}
\end{remark}
For simplicity, we assume $0 < \alpha \leq \beta < 1$. We introduce the following notation:
\begin{equation}\label{ob}
    \begin{aligned}
l_\beta &= \{(y,s)\in\mathbb{R}^2:\ y<0,\ s=\beta y\},\\
l_\alpha &= \{(y,s)\in\mathbb{R}^2:\ y\ge 0,\ s=\alpha y\},\\
\Omega_\beta &= \{(y,s)\in\mathbb{R}^2:\ \beta^{-1}y < s < \beta y\},\\
\tilde{\Omega}_\beta &= \{(y,s)\in\mathbb{R}^2:\ s<\beta y\ (y\le 0),\ \ s<-y\ (y>0)\},\\
\tilde{\Omega}_\alpha &= \{(y,s)\in\mathbb{R}^2:\ s<y\ (y\le 0),\ \ s<\alpha y\ (y>0)\}.
 \end{aligned}
\end{equation}
Additionally, define the directional derivatives aligned with the line $s=\beta y$:
$$
\partial_r = \partial_s + \beta \partial_y \quad \, 
\qquad 
\partial_\sigma = \beta \partial_s + \partial_y
$$
\begin{lemma}\label{lem5.5}
Assume A\ref{A1}--A\ref{A5}. Then there exists a sequence $\{\lambda_n\}$ such that 
$$
w_{\phi, \lambda_n} \to V_{\phi, \beta} \quad \text{and} \quad w_{\psi, \lambda_n} \to V_{\psi, \beta} \quad \text{as } \lambda_n \to \infty,
$$
locally uniformly in $\Omega'_0$, where 
$$
\Omega'_0 = \{(y, s) : s < \beta y\},
$$
$$
w_{\phi, \lambda_n}(y, s) = w_\phi(y - \lambda_n,\, s - \beta \lambda_n), \quad \text{and} \quad
w_{\psi, \lambda_n}(y, s) = w_\psi(y - \lambda_n,\, s - \beta \lambda_n).
$$
\end{lemma}
\begin{proof}
The translation $(y,s)\mapsto(y-\lambda, s-\beta\lambda)$ moves points along the tangent to $s=\beta y$.
Given any compact $K\subset\Omega'_0$, for all $\lambda$ large enough we have 
$K-(\lambda,\beta\lambda)\subset\tilde{\Omega}_0$,
so $w_{\phi,\lambda}, w_{\psi,\lambda}$ are well-defined on $K$.\\
Using the uniform bounds \eqref{eq:4.25}–\eqref{eq:4.32} for $w_{\phi},w_{\psi}$ and the fact that translations preserve these bounds, Arzelà–Ascoli provides a subsequence (still denoted $\lambda_n$) and limit functions $W_\phi,W_\psi\in \mathcal{C}^3(\Omega'_0)$ such that, on each compact $K\subset\Omega'_0$ and for $j=0,1,2,3$,
$$
D_r^j w_{\phi,\lambda_n} \to D_r^j W_\phi, 
\qquad 
D_\sigma^j w_{\phi,\lambda_n} \to D_\sigma^j W_\phi,
$$
and likewise for $w_{\psi,\lambda_n}$.\\
The pair $(w_{\phi,\lambda},w_{\psi,\lambda})$ solves the following system:
$$
\begin{cases}
\partial_s w_{\phi,\lambda} - \beta \partial_y w_{\phi,\lambda}
= 2^{-p}(w_{\phi,\lambda}+w_{\psi,\lambda})^p \;-\; \dfrac{\mu}{1+s-\beta\lambda}\,\dfrac{w_{\phi,\lambda}+w_{\psi,\lambda}}{2},\\[6pt]
\partial_s w_{\psi,\lambda} + \beta \partial_y w_{\psi,\lambda}
= 2^{-p}(w_{\phi,\lambda}+w_{\psi,\lambda})^p \;-\; \dfrac{\mu}{1+s-\beta\lambda}\,\dfrac{w_{\phi,\lambda}+w_{\psi,\lambda}}{2}.
\end{cases}
$$
Passing to the limit, as $\lambda \to \infty$, yields, in $\Omega'_0$,
\begin{equation}\label{7.9}
\begin{cases}
\partial_s W_\phi - \beta \partial_y W_\phi = 2^{-p}(W_\phi + W_\psi)^p, \\
\partial_s W_\psi + \beta \partial_y W_\psi = 2^{-p}(W_\phi + W_\psi)^p,
\end{cases}
\end{equation}
From \eqref{eq:4.25}--\eqref{eq:4.32} transferred to $w_{\phi,\lambda_n},w_{\psi,\lambda_n}$ and then to $W_\phi,W_\psi$, we get:
\begin{itemize}
\item $W_\phi, W_\psi \to 0$ and $D_r W_\phi, D_r W_\psi \to 0$ as $d_0(y,s)\to\infty$ in $\Omega'_0$;
\item the tangential monotonicity $\partial_\sigma w_\phi,\partial_\sigma w_\psi \ge 0$  is preserved by translation and therefore  $\partial_\sigma W_\phi,\partial_\sigma W_\psi \ge 0$ in $\Omega'_0$.
\end{itemize}
Let any compact $K\subset\Omega'_0$ and  $\lambda_n<\lambda_m$. By tangential monotonicity, we have
$$
w_{\phi,\lambda_n}(y,s) \;\le\; w_{\phi,\lambda_m}(y,s) \quad \text{on }K.
$$
Passing to the limits as $n,m\to\infty$ shows that $\partial_\sigma W_\phi\equiv 0$ on $K$. Since $K$ is arbitrary chosen, we obtain that
\begin{equation}\label{7.10}
\partial_\sigma W_\phi = \partial_\sigma W_\psi = 0 \quad \text{in } \Omega'_0.
\end{equation}
Hence, thanks to \eqref{7.10}, both $W_\phi,W_\psi$ depend only on the normal variable 
$\rho:=\beta y - s$ and solve 
\begin{equation}
    \left\{
    \begin{aligned}
    q(1+\beta) W_\phi(\rho) = 2^{-p}\big(W_\phi(\rho)+W_\psi(\rho)\big)^p, \\
q(1-\beta) W_\psi(\rho) = 2^{-p}\big(W_\phi(\rho)+W_\psi(\rho)\big)^p,
\end{aligned}
\right.
\end{equation}
Therefore, $W_\phi=V_{\phi,\beta}$ and $W_\psi=V_{\psi,\beta}$ in $\Omega'_0$, as claimed.
\end{proof}
Having established the convergence properties in Lemma \ref{lem5.5}, we now derive an important monotonicity result for the directional derivatives.
\begin{lemma}\label{c7.5}
Assume A\ref{A1}--A\ref{A5}. Then, we have
$$
\partial_\sigma w_\phi, \partial_\sigma w_\psi \geq 0 \quad \text{in } \tilde{\Omega}_0,
$$
where  $\tilde{\Omega}_0$ is defined in \eqref{tild}.
\end{lemma}

\begin{proof}
By Lemma \ref{lem5.5}, for every fixed $(y,s)\in\tilde{\Omega}_0$ the functions
$$
w_{\phi,\lambda}(y,s)=w_\phi(y-\lambda,s-\beta\lambda),\qquad
w_{\psi,\lambda}(y,s)=w_\psi(y-\lambda,s-\beta\lambda)
$$
converge (locally uniformly) as $\lambda\to\infty$ to the profiles $V_{\phi,\beta},V_{\psi,\beta}$. Hence,
$$
\partial_\sigma w_{\phi,\lambda}(y,s)=\partial_\sigma w_\phi(y-\lambda,s-\beta\lambda)\;\longrightarrow\;0,
\quad
\partial_\sigma w_{\psi,\lambda}(y,s)\;\longrightarrow\;0,
$$
as $\lambda\to\infty$.\\
Moreover, we have
$$
\partial_\sigma^2 w_\phi\ge 0,\qquad \partial_\sigma^2 w_\psi\ge 0 \quad \text{in }\tilde{\Omega}_0.
$$
Now, let \((y,s)\in\tilde{\Omega}_0\) and define $h_\phi(\lambda):=\partial_\sigma w_\phi(y-\lambda,s-\beta\lambda)$.\\ Then, $h_\phi'(\lambda)=-\partial_\sigma^2 w_\phi(y-\lambda,s-\beta\lambda)\le 0$, so $h_\phi$ is nonincreasing (in $\lambda$) and $\lim_{\lambda\to\infty}h_\phi(\lambda)=0$. Therefore, $h_\phi(0)=\partial_\sigma w_\phi(y,s)\ge 0$. 

The same argument applies to $w_\psi$. This ends the proof of this lemma.
\end{proof}
Building on these monotonicity properties, we can now establish a key comparison inequality between the limiting solutions.
\begin{lemma}\label{lem7.6}
Assume A\ref{A1}--A\ref{A5}. Then, we have
$$
V_{\phi, \beta} + V_{\psi, \beta} \geq w_\phi + w_\psi \quad \text{in } \Omega_\beta,
$$
where  $\Omega_\beta$ is defined in \eqref{ob}.
\end{lemma}

\begin{proof}
Recall that
$$
\partial_s w_\phi - \beta \partial_y w_\phi = 2^{-p}(w_\phi+w_\psi)^p,\qquad
\partial_s w_\psi + \beta \partial_y w_\psi = 2^{-p}(w_\phi+w_\psi)^p.
$$
Denote $S:=w_\phi+w_\psi$. A direct computation gives the identity
\begin{equation}\label{eq:tau-sigma-identity}
(\partial_\tau^2-\partial_\sigma^2)S
= (1-\beta^2)\,(\partial_s^2-\partial_y^2)S
= 2^{-p+1}p\,(1-\beta^2)\,S^{p-1}\,\partial_s S.
\end{equation}
Let $\varepsilon>0$. Now, define the regularized profiles
$$
V_{\phi,\beta}^\varepsilon(y,s)=C_{\phi,\beta}\,(\beta y-\varepsilon-s)^{-q},\qquad
V_{\psi,\beta}^\varepsilon(y,s)=C_{\psi,\beta}\,(\beta y-\varepsilon-s)^{-q},
$$
and set $S_\varepsilon:=V_{\phi,\beta}^\varepsilon+V_{\psi,\beta}^\varepsilon$. Then, $S_\varepsilon$ satisfies, in $\Omega_\beta^\varepsilon:=\{(y,s):\ s<\beta y-\varepsilon\}\subset\Omega_\beta$,
\begin{equation}\label{sub}
(\partial_\tau^2-\partial_\sigma^2)S_\varepsilon(y,s)
\;\le\;
2^{-p+1}p\,(1-\beta^2)\,S_\varepsilon^{\,p-1}(y,s)\,\partial_s S_\varepsilon(y,s).
\end{equation}
On the upper boundary $s=\beta y-\varepsilon$, we have $S_\varepsilon(y,s)\to \infty$ while $S(y,s)$ remains finite, hence
$$
S(y,s)-S_\varepsilon(y,s)\to-\infty\quad\text{on } \{s=\beta y-\varepsilon\}.
$$
As $|\beta y - s|\to\infty$ inside $\Omega_\beta^\varepsilon$, both $S$ and $S_\varepsilon$ vanish, and therefore
$$
\big(S-S_\varepsilon\big)(y,s)\to 0 \quad \text{as} \quad |\beta y - s|\to\infty.
$$
Assume by contradiction that $S-S_\varepsilon$ attains a positive local maximum at some $(y_e,s_e)\in\Omega_\beta^\varepsilon$.
Set
$$
\tau:=s+\beta y,\qquad \zeta:=s-\beta y,
$$
so that $s=\tfrac{\tau+\zeta}{2}$ and $y=\tfrac{\tau-\zeta}{2\beta}$, hence
$$
\partial_\tau=\tfrac12\partial_s+\tfrac{1}{2\beta}\partial_y,\qquad
\partial_\zeta=\tfrac12\partial_s-\tfrac{1}{2\beta}\partial_y.
$$
Consider the ray
$$
m_e:=\{(y,s)=(y_e,s_e)+\theta\,(1/\beta,1)\ :\ \theta\in\mathbb{R}\}.
$$
Along $m_e$, one has $\frac{d}{d\theta}\zeta=0$ and $\frac{d}{d\theta}\tau=2$, so the restriction
$\theta\mapsto F\big((y_e,s_e)+\theta(1/\beta,1)\big)$ (with $F:=S-S_\varepsilon$) is a one–variable function in the $\tau$–direction.
Along $m_e$, we have
$$
\frac{d}{d\theta}F=\tfrac1\beta\,\partial_y F+\partial_s F=2\,\partial_\tau F,\qquad
\frac{d^2}{d\theta^2}F=4\,\partial_\tau^2 F.
$$
Since $(y_e,s_e)$ is a local maximum of $F$, the restriction has a local maximum at $\theta=0$, hence
$$
\partial_\tau F(y_e,s_e)=0,\qquad \partial_\tau^2 F(y_e,s_e)\le 0.
$$
Moreover, $\nabla F(y_e,s_e)=0$, so
$$
\partial_\sigma F(y_e,s_e)=\tfrac12\partial_sF-\tfrac{1}{2\beta}\partial_yF=0.
$$

At $(y_e,s_e)$ we thus have
$$
\partial_\tau(S-S_\varepsilon)=0,\qquad \partial_\zeta(S-S_\varepsilon)=0,\qquad
\partial_\tau^2(S-S_\varepsilon)\le 0.
$$
Applying \eqref{eq:tau-sigma-identity} to $S$ and the subsolution inequality \eqref{sub} to $S_\varepsilon$, we obtain at $(y_e,s_e)$:
$$
\begin{aligned}\label{diff eq}
(\partial_\tau^2-\partial_\zeta^2)(S-S_\varepsilon)
&= 2^{-p+1}p(1-\beta^2)\Big(S^{p-1}\partial_s S - S_\varepsilon^{p-1}\partial_s S_\varepsilon\Big)\\
&= 2^{-p+1}p(1-\beta^2)\Big(\partial_s S\,(S^{p-1}-S_\varepsilon^{p-1})
+ S_\varepsilon^{p-1}\,\partial_s(S-S_\varepsilon)\Big).
\end{aligned}
$$
But at the local maximum $\partial_s(S-S_\varepsilon)(y_e,s_e)=0$, and since $S(y_e,s_e)>S_\varepsilon(y_e,s_e)$, the remaining term on the right-hand side in \eqref{diff eq} is positive. This contradicts the fact that, at a local maximum, the left–hand side in \eqref{diff eq} is nonpositive in the $\tau$–direction (we have $\partial_\tau^2(S-S_\varepsilon)\le 0$ and $\partial_\zeta(S-S_\varepsilon)=0$ there). Therefore $S\le S_\varepsilon$ in $\Omega_\beta^\varepsilon$. 

Finally, letting $\varepsilon\downarrow 0$ yields $w_\phi+w_\psi=S\le \lim_{\varepsilon\to 0}S_\varepsilon=V_{\phi,\beta}+V_{\psi,\beta}$ in $\Omega_\beta$.
\end{proof}

Having established the comparison between the solutions in Lemma \ref{lem7.6}, we now prove the exact equality.
\begin{lemma}\label{lem7.7}
Assume A\ref{A1}--A\ref{A5}. Then, we have
$$
w_\phi + w_\psi = V_{\phi, \beta} + V_{\psi, \beta} \quad \text{in } \Omega_\beta.
$$
\end{lemma}

\begin{proof}
Lemma \ref{c7.5} gives $\partial_\sigma w_\phi, \partial_\sigma w_\psi \ge 0$ in $\tilde{\Omega}_0$, whereas  $\partial_\sigma V_{\phi,\beta}=\partial_\sigma V_{\psi,\beta}=0$.\\ By Lemma \ref{lem5.5}, there exists a  sequence $\{\lambda_n\}_{n\ge 1}$
 such that
$$
w_{\phi,\lambda_n}\to V_{\phi,\beta},\qquad w_{\psi,\lambda_n}\to V_{\psi,\beta}\quad\text{locally uniformly in }\tilde{\Omega}_0.
$$
Passing to the limit along $\lambda_n$ in the monotonicity $\partial_\sigma w_\phi\ge 0$, $\partial_\sigma w_\psi\ge 0$ and using $\partial_\sigma V_{\cdot,\beta}\equiv 0$ yields the pointwise inequalities
\begin{equation}\label{7.12}
w_\phi \ge V_{\phi,\beta},\qquad w_\psi \ge V_{\psi,\beta}\quad\text{in }\tilde{\Omega}_0.
\end{equation}
Combining \eqref{7.12} with Lemma \ref{lem7.6} (which gives $V_{\phi,\beta}+V_{\psi,\beta}\ge w_\phi+w_\psi$ in $\Omega_\beta$) forces
$$
w_\phi+w_\psi=V_{\phi,\beta}+V_{\psi,\beta}\quad\text{in }\Omega_\beta,
$$
as claimed.
\end{proof}
With the sum equality established, we now prove the individual component equalities:
\begin{lemma}\label{lem7.8}
Assume A\ref{A1}--A\ref{A5}. Then, we have
$$
w_\phi = V_{\phi, \beta}, \quad w_\psi = V_{\psi, \beta} \quad \text{in } \tilde{\Omega}_\beta,
$$
where $\tilde{\Omega}_\beta$ is defined in \eqref{ob}.

\end{lemma}

\begin{proof}
Let $(y,s)\in \tilde{\Omega}_\beta\setminus \Omega_\beta$ and choose $(y_1,s_1),(y_2,s)\in\Omega_\beta$ such that the two characteristic segments from $(y,s)$ to $(y_1,s_1)$ have the slopes $\pm\beta$:
$$
\frac{s_1-s}{y_1-y_2}=\frac{1}{\beta}\quad\text{and}\quad \frac{s_1-s}{y_1-y}=-\,\frac{1}{\beta}.
$$
The functions $w_\phi,w_\psi$, $V_{\phi,\beta}$ and $V_{\psi,\beta}$ solve the  system
$$
\partial_s w_\phi-\beta\partial_y w_\phi=2^{-p}(w_\phi+w_\psi)^p,\qquad
\partial_s w_\psi+\beta\partial_y w_\psi=2^{-p}(w_\phi+w_\psi)^p,
$$
(and the same holds with $w$ being replaced by $V_\beta$). Thus, integrating along the characteristics gives the sum representation
$$
\begin{aligned}
(w_\phi+w_\psi)(y_1,s_1)
&=w_\phi(y,s)+w_\psi(y_2,s)
\\&\quad+\int_{0}^{s_1-s}\!2^{-p}\,(w_\phi+w_\psi)^p\!\bigl(y_1-\beta(s_1-s-r),\,s+r\bigr)\,dr
\\&\quad+\int_{0}^{s_1-s}\!2^{-p}\,(w_\phi+w_\psi)^p\!\bigl(y_1+\beta(s_1-s-r),\,s+r\bigr)\,dr,
\end{aligned}
$$
and an analogous identity stands with $w$ replaced by $V_\beta$.\\
Now, if $w_\phi(y,s)\neq V_{\phi,\beta}(y,s)$, then by \eqref{7.12} we have $w_\phi(y,s)>V_{\phi,\beta}(y,s)$.\\ Using also $w_\psi\ge V_{\psi,\beta}$ and the fact (from Lemma \ref{lem7.7}) that
$$
(w_\phi+w_\psi)^p=(V_{\phi,\beta}+V_{\psi,\beta})^p\quad\text{in }\Omega_\beta,
$$
the two characteristic formulas yield
$$
(w_\phi+w_\psi)(y_1,s_1)>(V_{\phi,\beta}+V_{\psi,\beta})(y_1,s_1),
$$
contradicting Lemma \ref{lem7.7}. Hence,
\begin{equation}\label{7.13}
w_\phi=V_{\phi,\beta}\quad\text{in }\tilde{\Omega}_\beta.
\end{equation}
Finally, from Lemma \ref{lem7.7}, we obtain the following identity
$$
w_\phi+w_\psi=V_{\phi,\beta}+V_{\psi,\beta}\quad\text{in }\tilde{\Omega}_\beta,
$$
that we combine with \eqref{7.13} yielding thus  $w_\psi=V_{\psi,\beta}$ in $\tilde{\Omega}_\beta$ as well.
\end{proof}
Now, we complete the proof of Theorem \ref{th:3.6} by combining our previous results.

\begin{proof}[Proof of Theorem \ref{th:3.6}]
Following the same methodology used for $\beta$, we can extract a subsequence that realizes the $\alpha$-slope (\textit{i.e.}, $
T_{0,\lambda}(y)\ \xrightarrow[\lambda\to 0^+]{}\ \alpha\,y,\qquad y>0.
$), and obtain the limiting functions
$$
w_\phi = V_{\phi, \alpha} \quad \text{and} \quad w_\psi = V_{\psi, \alpha} \quad \text{in } \tilde{\Omega}_\alpha,
$$
where, as before, $\tilde{\Omega}_\alpha := \{(y,s): s<\alpha y\}$.\\
If $\alpha \neq \beta$, then on the nonempty intersection
$$
\emptyset \neq \tilde{\Omega}_\alpha \cap \tilde{\Omega}_\beta
=\{(y,s): s<\min(\alpha y,\beta y)\},
$$
we would simultaneously have
$$
w_\phi = V_{\phi,\alpha}=C_{\phi,\alpha}(\alpha y-s)^{-q},
\quad\text{and}\quad
w_\phi = V_{\phi,\beta}=C_{\phi,\beta}(\beta y-s)^{-q},
$$
(and similarly for $w_\psi$). Since the functions $(\alpha y-s)^{-q}$ and $(\beta y-s)^{-q}$ are distinct on any open set when $\alpha\neq\beta$, equality on the open set $\tilde{\Omega}_\alpha\cap\tilde{\Omega}_\beta$ is impossible unless $\alpha=\beta$.\\
Therefore,
$$
\alpha=\beta,
$$
and, in particular,
$$
\inf_{\lambda_n} \frac{T_0(\lambda_n y)}{\lambda_n y}
= \alpha = \beta
= \sup_{\lambda_n} \frac{T_0(\lambda_n y)}{\lambda_n y}.
$$
This equality of infimum and supremum implies the linearity of $T_0$:
$$
T_0(y) = \alpha y, \quad \text{for all} \ \ y \in \mathbb{R}.
$$
Finally, combining the two characteristic identities for $(w_\phi+w_\psi)$ and $(V_{\phi,\beta}+V_{\psi,\beta})$  in Lemma~\ref{lem7.8} with  Lemma~\ref{lem7.7} and  \eqref{7.12}–\eqref{7.13}, and using the scaling limit $T_{0,\lambda}(y)\xrightarrow[\lambda\downarrow 0]{}\beta y$, we obtain $w_\phi=V_{\phi,\beta}$ and $w_\psi=V_{\psi,\beta}$ in $\tilde{\Omega}_\beta$. Hence, $T_0(y)=\beta y$ for all $y$.
This completes the proof of Theorem \ref{th:3.6}.
\end{proof}
\section{Proof of the Continuous Differentiability of the Blow-up Curve}
\label{section8}

This section is devoted to the proof of Theorem \ref{thm3.4}. We start by demonstrating that $T$ is differentiable on $B_{R^*}$ which is defined by \eqref{eq:ball}. For that purpose, we introduce, for the variables $(x,t)$, the following directional derivative in the direction of $\theta$:
$$
D_\theta := \cos\theta\,\partial_x + \sin\theta\,\partial_t,
$$
and for the rescaled variables $(y,s)$ we use the same notation, 
$$
D_\theta := \cos\theta\,\partial_y + \sin\theta\,\partial_s.
$$
If $\phi_\lambda(y,s)=\lambda^{q}\phi(x_0+\lambda y,T(x_0)+\lambda s)$, then
\begin{equation}\label{eq:scaling-Dtheta}
D_\theta \phi_\lambda(y,s)
= \lambda^{q+1}\, \bigl(D_\theta \phi\bigr)\bigl(x_0+\lambda y,\,T(x_0)+\lambda s\bigr),
\end{equation}
making the signs of $D_\theta\phi_\lambda$ and  $D_\theta\phi$ coincide under the blow-up scaling.

We now proceed by contradiction. Suppose, for that purpose, that there exists $x_0 \in B_{R^*}$ such that $T$ is not differentiable at $x_0$. Then, by the arguments in Section \ref{sec6.2} (mainly the fact that there exist two different limiting slopes from the non differentiability), there exist sequences $\{\lambda_n^{(1)}\}_{n\ge 1},$ $\{\lambda_n^{(2)}\}_{n\ge 1}
$ and constants $\alpha_1 < \alpha_2$ such that
$$
\phi_{\lambda_n}^{(1)} \to V_{\phi,\alpha_1}, \quad \text{and} \quad \phi_{\lambda_n}^{(2)} \to V_{\phi,\alpha_2},
$$
as $\lambda_n^{(1)}, \lambda_n^{(2)} \to 0$, locally uniformly in $\Omega_0$, with corresponding blow-up curves
$$
T_0^{(1)}(y) = \alpha_1 y, \quad \text{and} \quad T_0^{(2)}(y) = \alpha_2 y, \quad \text{for } y \in \mathbb{R}.
$$
Let $\theta_{\alpha_j} = \arctan \alpha_j$ for $j=1,2$. Without loss of generality, assume $0 \leq \theta_{\alpha_1} < \theta_{\alpha_2} < \pi$ and choose $0 < \varepsilon < \pi/2$ such that $\theta_{\alpha_1} + \varepsilon < \theta_{\alpha_2} - \varepsilon$.\\ For simplicity, we take $0 \leq \theta_{\alpha_1} < \theta_{\alpha_2} < \pi/2$.\\
A direct computation gives
\begin{equation}\label{eq:Vphi-Dtheta-sign}
\begin{aligned}
D_\theta V_{\phi,\alpha}(y,s)
&= q\,C_{\phi,\alpha}\,\bigl(\alpha y-s\bigr)^{-q-1}\,\bigl(\sin\theta-\alpha\cos\theta\bigr)
\\ &= q\,C_{\phi,\alpha}\sqrt{1+\alpha^2}\,\bigl(\alpha y-s\bigr)^{-q-1}\,\sin(\theta-\theta_\alpha).
\end{aligned}
\end{equation}
Hence $D_\theta V_{\phi,\alpha}>0$ whenever $\theta\in(\theta_\alpha,\theta_\alpha+\pi)$ and $s<\alpha y$.\\
Moreover, we have
\begin{equation}\label{eq:pi-antipodal}
D_{\theta+\pi}V_{\phi,\alpha}=-D_\theta V_{\phi,\alpha}.
\end{equation}
Define the angular sectors for $j=1,2$:
$$
S_j^\varepsilon = \{\theta_j : \theta_{\alpha_j} + \varepsilon < \theta_j < \theta_{\alpha_j} + \pi - \varepsilon\}.
$$
By \eqref{eq:Vphi-Dtheta-sign}, $D_{\theta_j}V_{\phi,\alpha_j}>0$ for every $\theta_j\in S_j^\varepsilon$.\\ Since the positivity margin in \eqref{eq:Vphi-Dtheta-sign} is uniform once we stay a fixed distance $\varepsilon$ away from the boundary angles and restrict to a small ball below both lines $s=\alpha_j y$, there exist $\varepsilon'>0$ and $\rho>0$ such that
$$
D_{\theta_1}V_{\phi,\alpha_1} > 2\varepsilon' \quad \text{and} \quad D_{\theta_2}V_{\phi,\alpha_2} > 2\varepsilon',
$$
in $\Omega_{1,0} \cap \Omega_{2,0} \cap B_\rho(0,0)$, where
$$
\Omega_{j,0} = \{(y,s) \in \mathbb{R}^2 : s < \alpha_j y\}, \quad j=1,2,
$$
and
$$
B_\rho(y',s') = \{(y,s) \in \mathbb{R}^2 : \sqrt{(y-y')^2 + (s-s')^2} < \rho\}.
$$
Furthermore, by the $\mathcal{C}^1_{\mathrm{loc}}-$convergence in the scaling variables, there exists $n_0 \in \mathbb{N}$ such that
$$
\Omega_{\lambda_{n_0}}^{(1)} \cap \Omega_{\lambda_{n_0}}^{(2)} \cap \Omega_{1,0} \cap \Omega_{2,0} \cap B_\rho(0,0) \neq \emptyset,
$$
and
$$
|D_{\theta_1}\phi_{\lambda_{n_0}}^{(1)} - D_{\theta_1}V_{\phi,\alpha_1}| \leq \varepsilon', \quad |D_{\theta_2}\phi_{\lambda_{n_0}}^{(2)} - D_{\theta_2}V_{\phi,\alpha_2}| \leq \varepsilon'.
$$
Here
$$
\Omega_\lambda = \{(y,s) \in \mathbb{R}^2 : (x_0 + \lambda y, T(x_0) + \lambda s) \in \Omega\}.
$$
Therefore,  the $\mathcal{C}^1-$convergence of $\phi_{\lambda_n^{(j)}}$ to $V_{\phi,\alpha_j}$, on any compact subset of $\Omega_{j,0}$, follows from Section~\ref{sec6.2}. Indeed, the nonempty intersection of the domains holds for $n_0$ large since all sets are open and the lines $s=\alpha_j y$ pass through the origin.\\
Using \eqref{eq:scaling-Dtheta}, we deduce that 
\begin{equation}\label{6.1}
D_\theta\phi > 0, \quad \text{for } \theta \in S_\varepsilon^{(1)} \cup S_\varepsilon^{(2)},
\end{equation}
inside the subset 
$$
\Omega_{x_0,n_0}^* = \Omega \cap B_{\min\{\lambda_{n_0}^{(1)}, \lambda_{n_0}^{(2)}\}\rho}(x_0,T(x_0)) \cap \Omega_1^* \cap \Omega_2^*,
$$
where $\Omega_j^* = \{(x,t) \in \mathbb{R}^2 : t - T(x_0) < \alpha_j(x - x_0)\}$ for $j=1,2$.\\
In particular, for $\theta^* \in (\theta_{\alpha_1} + \varepsilon, \theta_{\alpha_2} - \varepsilon) \subset S_\varepsilon^{(1)}$, we have
\begin{equation}\label{6.2}
D_{\theta^*}\phi > 0 \quad \text{in } \Omega_{x_0,n_0}^*.
\end{equation}
Moreover, since $\theta^* + \pi \in (\theta_{\alpha_1} + \pi + \varepsilon, \theta_{\alpha_2} + \pi - \varepsilon) \subset S_\varepsilon^{(2)}$,
\begin{equation}\label{6.3}
D_{\theta^* + \pi}\phi > 0 \quad \text{in } \Omega_{x_0,n_0}^*.
\end{equation}
However, \eqref{6.2} and \eqref{6.3} contradict the identity $D_{\theta^*}\phi = -D_{\theta^* + \pi}\phi$ in $\Omega_{x_0,n_0}^*$, which follows from the definition of $D_\theta$.\\ Therefore, $T$ must be differentiable at $x_0 \in B_{R^*}$.

We now establish the continuity of $T'$ in $B_{R^*}$ by contradiction.\\ Assume there exists $x_0 \in B_{R^*}$ where $T'$ is discontinuous. Let $\alpha_{x_0} = T'(x_0)$ and define the local domain:
$$
\Omega_{x_0,n}^{**} = \Omega \cap B_{\lambda_n\rho}(x_0,T(x_0)) \cap \Omega_{x_0}^{**},
$$
where 
$$
\Omega_{x_0}^{**} = \{(x,t) \in \mathbb{R}^2 : t - T(x_0) < \alpha_{x_0}(x - x_0)\}.
$$
By the discontinuity assumption, there exists $x_1 \in B_{R^*}$ (with $\alpha_{x_1} = T'(x_1)$) and $\varepsilon' > 0$ with $n_0, n_1 \in \mathbb{N}$ such that:
\begin{equation}\label{6.4}
\begin{split}
& 0<\varepsilon'<\frac{\pi}{2},\\
& \Omega_{x_0,n_0}^{**}\cap \Omega_{x_1,n_1}^{**}\neq \emptyset,\\
& \big|\theta_{\alpha_{x_1}}-\theta_{\alpha_{x_0}}\big|>2\varepsilon'.
\end{split}
\end{equation}
We choose $n_0,n_1$ large so that
$$
\Omega_{x_0,n_0}^{**}\cap \Omega_{x_1,n_1}^{**}\neq\emptyset,
$$
and $\mathcal{C}^1_{\text{loc}}$–closeness holds on the intersection:
$$
\bigl|D_{\theta_{\alpha_{x_0}}}\phi_{\lambda_{n_0}}^{(0)}-D_{\theta_{\alpha_{x_0}}}V_{\phi,\alpha_{x_0}}\bigr|\le \varepsilon',
\quad
\bigl|D_{\theta_{\alpha_{x_1}}}\phi_{\lambda_{n_1}}^{(1)}-D_{\theta_{\alpha_{x_1}}}V_{\phi,\alpha_{x_1}}\bigr|\le \varepsilon'.
$$
Here $\theta_{\alpha_{x_j}} = \arctan \alpha_{x_j}$ for $j=0,1$. Define the angular sectors for each point:
$$
S_{\varepsilon',x_j} = \{\theta_j : \theta_{\alpha_{x_j}} + \varepsilon' < \theta_j < \theta_{\alpha_{x_j}} + \pi - \varepsilon'\}, \quad j=0,1.
$$
Arguing similarly as the differentiability proof  (using the similarity limit with slope $\alpha_{x_j}$ and the $\mathcal{C}^1_{\mathrm{loc}}$ convergence in scaling coordinates around $(x_j,T(x_j))$), we obtain for each $j=0,1$:
$$
D_{\theta_j}\phi > 0 \quad \text{in} \quad \Omega_{x_j,n_j}^{**} \quad \text{for all } \theta_j \in S_{\varepsilon',x_j}.
$$
Assume without loss of generality that $0 \leq \theta_{\alpha_{x_0}} < \theta_{\alpha_{x_1}} < \pi/2$.\\ From \eqref{6.4}, we have
$$
\theta_{\alpha_{x_0}} + \varepsilon' < \theta_{\alpha_{x_1}} - \varepsilon'.
$$
Choosing any $\tilde{\theta}$ in the interval $(\theta_{\alpha_{x_0}} + \varepsilon', \theta_{\alpha_{x_1}} - \varepsilon')$ yields the following:
\begin{itemize}
\item[(i)] If  $\tilde{\theta} \in S_{\varepsilon',x_0}$ then $D_{\tilde{\theta}}\phi > 0$ in $\Omega_{x_0,n_0}^{**}$,
\item[(ii)] If   $\tilde{\theta} + \pi \in S_{\varepsilon',x_1}$ then $ D_{\tilde{\theta}+\pi}\phi > 0$ in $\Omega_{x_1,n_1}^{**}$.
\end{itemize}
On the nonempty intersection $\Omega_{x_0,n_0}^{**}\cap \Omega_{x_1,n_1}^{**}$, both inequalities hold simultaneously, which contradicts the fundamental identity,
$$
D_{\tilde{\theta}+\pi}\phi = -D_{\tilde{\theta}}\phi \quad \text{in} \quad \Omega.
$$
Hence, this implies that $T'$ must be continuous at every $x_0 \in B_{R^*}$.

This completes the proof of $\mathcal{C}^1-$regularity for the blow-up curve $T$, and Theorem \ref{thm3.4} is thus proven.
\section*{Acknowledgments}
 This work was initiated in 2023 when Takiko Sasaki visited the University Sorbonne Paris Nord, as an Invited Professor of the ``\'Ecole Universitaire de Recherche'' (EUR) ``Paris Nord Graduate School in Mathematics and Computer Science'' (PNGS-M$\And$CS), with support from the French State Program "Investissement d’Avenir", managed by the "Angence Nationale de la Recherche" under the grant ANR-18-EURE-0024. She wishes to thank the EUR PNGS-M$\And$CS together the Laboratoire Analyse, Géométrie et Applications (LAGA) for their hospitality.\\
A part of this work was carried out during the visit of Firas Kaabi to LAGA in 2024. He would like to thank LAGA for its hospitality and support during this visit.\\
Hatem Zaag wishes to thank Pierre Raphaël and
the ”SWAT” ERC project for their support.

\section*{Declarations}
\subsection*{Ethical Statement}
This study was conducted in accordance with ethical guidelines and regulations. The research did not involve any human participants, human data, or animals. 
\subsection*{Competing interests} The authors declare no competing interests.
\section*{Consent for Publication}
\subsection*{Funding} Not Applicable.
\subsection*{Data Availability Statement}
All data underlying the results are available as part of the article and no additional source data are required.

\end{document}